\documentclass[a4paper,12pt]{amsart}
\usepackage{amsmath}
\usepackage{amsmath}
\usepackage{amssymb}
\usepackage{amscd}
\usepackage{mathrsfs}
\usepackage[all]{xy}
\usepackage[dvipdfm]{graphicx}
\def\mathcal{\mathscr}
\everymath{\displaystyle}
\newtheorem{thm}{Theorem}[section]
\newtheorem{lem}[thm]{Lemma}
\newtheorem{cor}[thm]{Corollary}
\newtheorem{prop}[thm]{Proposition}
\newtheorem{conj}[thm]{Conjecture}
\theoremstyle{definition}

\newtheorem{rem}[thm]{Remark}
\newtheorem{claim}[thm]{Claim}

\newtheorem{ex}[thm]{Example}

\newcommand{\mca}[1]{{\mathcal{#1}}}
\def\ad{\text{\rm ad}}

\def\Z{{\mathbb Z}}
\def\C{{\mathbb C}}

\def\R{{\mathbb R}}

\def\len{\text{\rm len}}

\def\CF{\text{\rm CF}}

\def\CrV{\text{\rm CrV}}
\def\CrP{\text{\rm CrP}}

\def\CM{\text{\rm CM}}
\def\CZ{\text{\rm CZ}}

\def\ev{\text{\rm ev}}

\def\HF{\text{\rm HF}}
\def\HM{\text{\rm HM}}

\def\hor{\text{\rm hor}}
\def\HZ{\text{\rm HZ}}
\def\id{\text{\rm id}}
\def\Im{\text{\rm Im}\,}
\def\ind{\text{\rm ind}}

\def\interior{\text{\rm int}}

\def\lin{\text{\rm lin}}

\def\neg{\text{\rm neg}}
\def\Per{\text{\rm Per}\,}

\def\Spec{\text{\rm Spec}\,}
\def\supp{\text{\rm supp}}

\def\ver{\text{\rm ver}}

\begin{document}
\pagestyle{plain}
\thispagestyle{plain}

\title[Hofer-Zehnder capacity of unit disk cotangent bundles and the loop product]
{Hofer-Zehnder capacity of unit disk cotangent bundles and the loop product}

\author[Kei Irie]{Kei Irie}
\address{Department of Mathematics, Faculty of Science, Kyoto University,
Kyoto 606-8502, Japan}
\email{iriek@math.kyoto-u.ac.jp}

\subjclass[2010]{Primary:53D40, Secondary:70H12}
\date{\today}

\begin{abstract}
We prove the finiteness of the Hofer-Zehnder capacity of unit disk cotangent bundles of closed Riemannian manifolds, 
under some simple topological assumptions on the manifolds. 
The key ingredient of the proof is a computation of the pair-of-pants product on Floer homology of cotangent bundles. 
We reduce it to a simple computation of the loop product, making use of results of 
A. Abbondandolo - M. Schwarz.
\end{abstract}

\maketitle

\section{Introduction}
\subsection{Main result}
This paper concerns the Hofer-Zehnder capacity of unit disk cotangent bundles of closed Riemannian manifolds. 
First we recall the definition of the Hofer-Zehnder capacity, following \cite{HZ}.
Let $(X,\omega)$ be a symplectic manifold, possibly with boundary. 
For any Hamiltonian $H \in C^\infty(X)$, its \textit{Hamiltonian vector field} $X_H \in \mca{X}(X)$ is defined by 
$\omega(X_H, \, \cdot \,) = -dH(\,\cdot\,)$. 
$H \in C^\infty(X)$ is called \textit{nice} if and only if it satisfies the following conditions:
\begin{itemize}
\item There exists a compact set $K \subset \interior X:= X \setminus \partial X$ such that 
$H \equiv \max H$ on $X \setminus K$. 
\item There exists a nonempty open set $U \subset X$ such that $H \equiv 0$ on $U$. 
\item $H(x) \ge 0$ for any $x \in M$.
\item Any nonconstant periodic orbit of $X_H$ has a period strictly larger than $1$.
\end{itemize}
Then, the Hofer-Zehnder capacity $c_{\HZ}(X,\omega)$ is defined as 
\[
c_{\HZ}(X,\omega):= \sup \bigl\{ \max H \bigm{|} \text{$H$ is nice} \bigr\}.
\]

To explain our main result and the idea of its proof, we fix some notations. 
Let $M$ be a closed oriented Riemannian manifold. 

\begin{itemize}
\item $\Lambda_M$ denotes the Hilbert manifold of free loops $S^1(:=\R/\Z) \to M$ of Sobolev class $W^{1,2}$ (i.e. 
loops with square integrable derivatives).
\item A homotopy class of free loops on $M$ is \textit{trivial} if it consists of contractible loops. 
\item Let $\alpha$ be a homotopy class of free loops on $M$, and $a>0$. Let us define 
\begin{align*}
\Lambda^{\alpha}_M&:= \{ \gamma \in \Lambda_M \mid [\gamma]=\alpha\}, \qquad
\Lambda^{<a}_M:= \{ \gamma \in \Lambda_M \mid \len(\gamma)<a \}, \\
\Lambda^{<a,\alpha}_M&:= \Lambda^{<a}_M \cap \Lambda^{\alpha}_M.
\end{align*}
$\len(\gamma)$ denotes the length of $\gamma$.
Obviously, $\Lambda^{<a}_M, \Lambda^\alpha_M$ are open sets in $\Lambda_M$.
\item $DT^*M$ denotes the unit disk cotangent bundle of $M$, and $ST^*M$ denotes its boundary, i.e.
\[
DT^*M:=\bigl\{(q,p) \in T^*M \bigm{|} |p| \le 1 \bigr\}, \quad
ST^*M:=\bigl\{(q,p) \in T^*M \bigm{|} |p|=1 \bigr\}.
\]
\item $\pi_M$ denotes the natural projection map $T^*M \to M; (q,p) \mapsto q$. 
\item $\lambda_M \in \Omega^1(T^*M)$, $\omega_M \in \Omega^2(T^*M)$, $Z_M \in \mca{X}(T^*M)$ is defined as
\begin{align*}
\lambda_M(\xi)&:= p(d\pi_M(\xi))\, \bigl((q,p) \in T^*M, \, \xi \in T_{(q,p)}(T^*M) \bigr), \\
\omega_M&:=d\lambda_M, \\
i_{Z_M}\omega_M&:= \lambda_M.
\end{align*}
\item For any periodic orbit $\gamma$ of a certain vector field, its period is denoted as $\Per(\gamma)$. 
\end{itemize}

Our main result is the following: 

\begin{thm}\label{thm:main}
Let $M$ be a closed oriented Riemannian manifold. 
Let $\alpha$ be a nontrivial homotopy class of free loops on $M$. 
Suppose that the evaluation map 
$\ev: \Lambda^{\alpha}_M \to M; \gamma \mapsto \gamma(0)$ has a smooth section $s$
(i.e. $s: M \to \Lambda^{\alpha}_M$ such that $\ev \circ s = \id_M$).
Then $c_{\HZ}(DT^*M, \omega_M) < \infty$. 
\end{thm} 

The following corollary is immediate from Theorem \ref{thm:main}.

\begin{cor}\label{cor:circle}
Let $M$ be a closed oriented Riemannian manifold.
Suppose that $M$ admits a smooth $S^1$ action $S^1 \times M \to M$ 
such that $S^1$ orbit $\gamma_p: S^1 \to M; t \mapsto t \cdot p$ is not contractible for any $p \in M$. 
Then $c_{\HZ}(DT^*M,\omega_M) <\infty$. 
\end{cor}
\begin{proof}
We may assume that $M$ is connected. 
Set $\alpha:=[\gamma_p]$.
Then, 
$s: M \to \Lambda^\alpha_M; p \mapsto \gamma_p$ 
satisfies the assumption of Theorem \ref{thm:main}.
\end{proof}

\subsection{Idea of the proof}
We briefly explain an idea to prove Theorem \ref{thm:main}.
We use some terminologies without their definitions.
Their definitions are explained in section 2. 

For any homotopy class $\alpha$ of free loops on $M$, one can define 
\textit{Floer homology} $\HF_*^{\alpha}(DT^*M)$ of a Liouville domain $(DT^*M, \lambda_M)$. 
$\bigoplus_{\alpha} \HF_*^{\alpha}(DT^*M)$ (where $\alpha$ runs over all homotopy classes of free loops) is abbreviated as $\HF_*(DT^*M)$. 
It satisfies the following properties ($n:=\dim M$):
\begin{itemize}
\item There exists a natural homomorphism 
$\iota_\infty: H^{n-*}(DT^*M) \to \HF_*(DT^*M)$. $\iota_\infty(1) \in \HF_n(DT^*M)$ is denoted as $F_\infty$. 
\item There exists a natural product structure (the \textit{pair-of-pants product}) on Floer homology: 
$\HF_i(DT^*M) \otimes \HF_j(DT^*M) \to \HF_{i+j-n}(DT^*M); x \otimes y \mapsto x*y$.
\end{itemize}

The following facts are established in \cite{AS1}, \cite{AS2}, \cite{SW}, \cite{Viterbo}:
\begin{enumerate}
\item[(a):] For any $\alpha$, there exists a natural isomorphism $\HF^{\alpha}_*(DT^*M) \cong H_*(\Lambda^{\alpha}_M)$. 
\item[(b):] The isomorphism in (a) 
 interwines the pair of pants product on $\HF_*(DT^*M)$ with the \textit{loop product} on $H_*(\Lambda_M)$. 
\end{enumerate}
In section 3, we recall the definition of the loop product, and state the above results in more precise form (Theorem \ref{thm:loop}).

The key ingredient of the proof is a simple computation of the loop product (Lemma \ref{lem:key}).
Combined with the facts (a), (b), 
Lemma \ref{lem:key} implies the following (Corollary \ref{cor:key}): 
\begin{enumerate}
\item[(c):] Under the asssumptions of Theorem \ref{thm:main}, there exist $x \in \HF_n^{\alpha}(DT^*M)$ and 
$y \in \HF_n(DT^*M)$ such that $x*y = F_\infty \in \HF_n(DT^*M)$.
\end{enumerate}
Theorem \ref{thm:main} is proved by 
(c) and 
a criterion for Hamiltonians to have nonconstant periodic orbits (Proposition \ref{prop:criterion}), 
which is based on the theory of spectral invariants. 

\subsection{Organization of the paper}
In section 2, we recall Floer theory on Liouville domains. 
We define truncated Floer homology of Liouville domains, and define the pair-of-pants product on Floer homology. 
Section 3 concerns Floer homology of cotangent bundles. We recall the definition of the loop product, and state the above facts (a), (b) 
in more precise form 
(Theorem \ref{thm:loop}). 
Although Theorem \ref{thm:loop} is essentially established in \cite{AS1} and \cite{AS2}, it requires some technical arguments to deduce it 
from results in those papers in a rigorous manner. 
These arguments are rather technical, hence they are postponed until section 6.
In section 4, we introduce the notion of spectral invariants, and establish a criterion 
for Hamiltonians to have nonconstant periodic orbits (Proposition \ref{prop:criterion}).
In section 5, we prove our main result, Theorem \ref{thm:main}.
In section 6, we prove Theorem \ref{thm:loop}.
In section 7, we discuss a quantitative refinement of our main result. 

\textbf{Acknowledgement.}
The author would like to appreciate his advisor professor Kenji Fukaya
for precious comments. 

\section{Floer theory on Liouville domains}
In this section, we recall Floer theory on Liouville domains. 
In section 2.1, we recall basic objects (Liouville domains, Hamiltonians, almost complex structures) and 
prove some convexity results for solutions of the Floer equations (Lemma \ref{lem:convexity}, Lemma \ref{lem:convexity3}).
In section 2.2, 
first we define truncated Floer homology of (admissible) Hamiltonians. 
After that, we define truncated Floer homology of Liouville domains. 
In section 2.3, we define the pair-of-pants product on truncated Floer homology. 

\subsection{Preliminaries}
\subsubsection{Liouville domains}

\textit{Liouville domain} is a pair $(X, \lambda)$, where $X$ is a $2n$-dimensional compact manifold with boundary, $\lambda \in \Omega^1(X)$ such that 
$d\lambda$ is a symplectic form on $X$, and 
$\lambda \wedge (d\lambda)^{n-1}>0$ on $\partial X$. 
\textit{Liouville vector field} $Z \in \mca{X}(X)$ is defined as $i_Z(d\lambda)=\lambda$. 
It is easy to show that $Z$ points strictly outwards on $\partial X$.
For any Liouville domain $(X,\lambda)$, 
$(\partial X, \lambda)$ is a contact manifold. 
We define $\Spec(X,\lambda)$ by 
\[
\Spec (X,\lambda):= \biggl\{ \int_\gamma \lambda \biggm{|} \text{$\gamma$ is a periodic Reeb orbit on $(\partial X, \lambda)$} \biggr\}.
\]
Obviously, $\Spec(X,\lambda) \subset (0,\infty)$. Moreover, 
it is well-known that $\Spec (X,\lambda)$ is closed and nowhere dence in $\R$. 

Let $I: \partial X \times (0,1] \to X$ be an embedding defined by 
\[
I(z,1)=z, \qquad \partial_rI(z,r) = r^{-1}Z\bigl(I(z,r)\bigr).
\]
It is easy to check that $I^*\lambda(z,r)=r \lambda(z)$ for any $(z,r) \in \partial X \times (0,1]$.

Define a manifold $\hat{X}$ by 
\[
\hat{X}:= X \cup_I \partial X \times (0,\infty), 
\]
and $\hat{\lambda} \in \Omega^1(\hat{X})$ by 
\[
\hat{\lambda}(x):= \begin{cases}
                   \lambda(x) &( x \in X), \\
                   r\lambda(z) &\bigl( x= (z,r) \in \partial X \times (0,\infty) \bigr). 
                   \end{cases}
\]
$(\hat{X}, \hat{\lambda})$ is called the \textit{completion} of $(X,\lambda)$. 
$d\hat{\lambda}$ is a symplectic form on $\hat{X}$.
For each $r>0$, $X(r)$ denotes the bounded domain in $\hat{X}$ with boundary $\partial X \times \{r\}$, i.e.
\[
X(r):=\begin{cases}
      X \cup \partial X \times [1,r] &( r \ge 1), \\
      X \setminus \partial X \times (r,1] &( r<1).
      \end{cases}
\]

\begin{ex}\label{ex:TM}
When $M$ is a closed Riemannian manifold, $(DT^*M, \lambda_M)$ is a Liouville domain.
There exists a unique diffeomorphism $\varphi: \widehat{DT^*M} \to T^*M$, such that $\varphi^*\lambda_M = \hat{\lambda_M}$ and 
$\varphi|_{DT^*M}$ is the inclusion $DT^*M \to T^*M$. 
Hence we identify $(T^*M,\lambda_M)$ with the completion of $(DT^*M,\lambda_M)$.
\end{ex}

\subsubsection{Hamiltonians}

For $H \in C^\infty(S^1 \times \hat{X})$, $H_t \in C^\infty(\hat{X})$ is defined by $H_t(x):=H(t,x)$. 
$\mca{P}(H)$ denotes the set of $1$-periodic orbits of $(X_{H_t})_{t \in S^1}$, i.e.
\[
\mca{P}(H):= \{x: S^1 \to \hat{X} \mid X_{H_t}(x(t))= \partial_tx(t) \}.
\]
$H$ is \textit{nondegenerate} when all orbits in $\mca{P}(H)$ are nondegenerate.
$H$ is \textit{linear at $\infty$} when there exist $a_H>0$, $b_H \in \R$ and $r_0 \ge 1$ such that 
$H_t(z,r)=a_Hr+b_H$ for any $t \in S^1$, $z \in \partial X$, $r \ge r_0$.
$H$ is \textit{admissible} when it is nondegenerate and linear at $\infty$.
We denote the set of admissible Hamiltonians by $\mca{H}_{\ad}(X,\lambda)$.
Notice that any $H \in \mca{H}_{\ad}(X,\lambda)$ satisfies 
$a_H \notin \Spec(X,\lambda)$, since otherwise $\mca{P}(H)$ contains infinitely many degenerate orbits.

\subsubsection{Almost complex structures}

Let $J$ be an almost complex structure on $\hat{X}$. 
$J$ is \textit{compatible with $d\hat{\lambda}$} when 
\[
g_J: TM \otimes TM \to \R; \qquad v \otimes w \mapsto d\hat{\lambda}(v,Jw)
\]
is a Riemannian metric (we denote $g_J(v,v)^{1/2}$ as $|v|_J$).
$\mca{J}(\hat{X},\hat{\lambda})$ denotes the set of almost complex structures on $\hat{X}$, which are compatible with $d\hat{\lambda}$.

Let $I \subset (0,\infty)$ be an interval. 
$J$ is \textit{of contact type on} $\partial X \times I$, when 
$dr \circ J(z,r)= -\lambda(z)$ for any $(z,r) \in \partial X \times I$.
If $J$ is of contact type on $\partial X \times (r_0,\infty)$ for some $r_0$, 
$J$ is \textit{of contact type at} $\infty$.
A family of almost complex structures $(J_a)_{a \in A}$ is of contact type on $\partial X \times I$ 
when each $J_a$ is of contact type on $\partial X \times I$.

\subsubsection{Convexity results}
We prove the following convexity results, which are necessary to develop Floer theory 
on Liouville domains.

\begin{lem}\label{lem:convexity}
Let $(X,\lambda)$ be a Liouville domain, 
$(H_{s,t})_{(s,t) \in \R \times S^1}$ be a family of Hamiltonians on $\hat{X}$, and 
$(J_{s,t})_{(s,t) \in \R \times S^1}$ be a family of elements in $\mca{J}(\hat{X},\hat{\lambda})$. 
Suppose that there exists $r_0 >0$, such that the following holds:
\begin{itemize}
\item There exist $a,b \in C^\infty(\R)$ such that 
$H_{s,t}(z,r)=a(s)r+b(s)$ on $\partial X \times [r_0,\infty)$, and $a'(s) \ge 0$ for any $s \in \R$.
\item For any $(s,t) \in \R \times S^1$, $J_{s,t}$ is of contact type on $\partial X \times [r_0, \infty)$.
\end{itemize}
Under these assumptions, if $u: \R \times S^1 \to \hat{X}$ satisfies the Floer equation 
$\partial_s u - J_{s,t}(\partial_t u- X_{H_{s,t}}(u)) = 0$ and 
$u^{-1}\bigl(\partial X \times (r_0, \infty) \bigr)$ is bounded, 
then $u(\R \times S^1) \subset X(r_0)$.
\end{lem}
\begin{proof}
If $u(\R \times S^1)$ is not contained in $X(r_0)$, then there exists $r_1>r_0$ such that 
$u(\R \times S^1)$ is not contained in $X(r_1)$, and $u$ is transversal to $\partial X \times \{r_1\}$.
Then, $D:=u^{-1}\bigl(\partial X \times [r_1,\infty) \bigr)$ is a compact surface with boundary, and 
\[
0 < \int_D |\partial_s u|_{J_{s,t}}^2 ds dt =\int_D d\hat{\lambda}(\partial_t u - X_{H_{s,t}}(u), \partial_s u)\, dsdt 
= \int_D dH_{s,t}(\partial_s u)\, dsdt - u^*(d\hat{\lambda}).
\]
On the otherhand, if $u(s,t) \in \partial X \times [r_0,\infty)$, 
\[
dH_{s,t}(\partial_s u) =a(s) \partial_sr(s,t)
\le a(s)\partial_sr(s,t) + a'(s) r(s,t)
= \partial_s(a(s)r(s,t))
= \partial_s\bigl(\hat{\lambda}(X_{H_{s,t}}(u)) \bigr).
\]
Hence we get
\[
\int_D dH_{s,t}(\partial_s u)\, dsdt - u^*(d\hat{\lambda})
\le \int_D \partial_s\bigl(\hat{\lambda}(X_{H_{s,t}}(u)) \bigr)\, ds dt - u^*(d\hat{\lambda})
= \int_{\partial D} \hat{\lambda}(X_{H_{s,t}}(u) dt - du). 
\]
We can compute the right hand side as follows
($j$ denotes the almost complex structure on $\R \times S^1$ which is defined by $j(\partial_s)=\partial_t$):
\[
\int_{\partial D} \hat{\lambda}(X_{H_{s,t}} (u)\,dt - du) 
= \int_{\partial D} \hat{\lambda}(J_{s,t} \circ (X_{H_{s,t}}(u)\, dt - du)) \circ j  
= r_1 \int_{\partial D} dr(X_{H_{s,t}}(u)\,dt - du) \circ j.
\]
The first equality follows from the Floer equation, 
the second equality holds since $J_{s,t}$ is of contact type on $\partial X \times [r_0,\infty)$ and
$u(\partial D) \subset \partial X \times \{r_1\}$.
Finally, 
\[
\int_{\partial D} dr(X_{H_{s,t}}(u)\,dt - du) \circ j <0.
\]
This is because $dr(X_{H_{s,t}})=0$ on $\partial X \times \{r_1\}$, and 
$dr\bigl(du(jV)\bigr)>0$ when $V$ is a vector tangent to $\partial D$, positive with respect to the boundary orientation.
Hence we get a contradiction.
\end{proof}

\begin{lem}\label{lem:convexity3}
Let $(X,\lambda)$ be a Liouville domain, 
$(H_{s,t})_{(s,t) \in \R \times S^1}$ be a family of Hamiltonians on $\hat{X}$, and 
$(J_{s,t})_{(s,t) \in \R \times S^1}$ be a family of elements in $\mca{J}(\hat{X},\hat{\lambda})$. 
Suppose that $\partial_sH_{s,t}(x) \ge 0$ for any $(s,t) \in \R \times S^1$ and $x \in \hat{X}$, and there exist $0<r_0<r_1$ with the
following properties:
\begin{itemize}
\item There exist $a,b \in C^\infty(\R)$ such that $H_{s,t}(z,r)=a(s)r+b(s)$ on $\partial X \times [r_0,r_1]$, and $b'(s) \le 0$ for any $s \in \R$.
\item For any $(s,t) \in \R \times S^1$, $J_{s,t}$ is of contact type on $\partial X \times [r_0,r_1]$.
\end{itemize}
Under these assumptions, if $u:\R \times S^1 \to \hat{X}$ satisfies the Floer equation $\partial_s u-J_{s,t}(\partial_t u - X_{H_{s,t}}(u))=0$
and $u^{-1}(\partial X \times (r_0,\infty))$ is bounded, then $u(\R \times S^1) \subset X(r_0)$.
\end{lem}
\begin{proof}
If $u(\R \times S^1)$ is not contained in $X(r_0)$, there exists $r_2 \in (r_0,r_1)$ such that 
$u(\R \times S^1)$ is not contained in $X(r_2)$, and $u$ is transversal to $\partial X \times \{r_2\}$. 
Then, $D:=u^{-1}\bigl(\partial X \times [r_2,\infty)\bigr)$ is a compact surface with boundary, and 
\[
0< \int_D |\partial_s u|_{J_{s,t}}^2 + \partial_sH_{s,t}(u) \, dsdt
=\int_{\partial D} H_{s,t}(u)\,dt - u^*\hat{\lambda}.
\]
When $(s,t) \in \partial D$, 
$H_{s,t}\bigl(u(s,t)\bigr)=a(s)r_2+b(s)$ and 
$\hat{\lambda}(X_{H_{s,t}}(u(s,t))\bigr)=a(s)r_2$.
Therefore
\[
\int_{\partial D} \hat{\lambda}(X_{H_{s,t}}(u)) - H_{s,t}(u)\,dt = -\int_{\partial D} b(s)\,dt 
=-\int_D b'(s)\, dsdt \ge 0.
\]
Hence we get $0 < \int_{\partial D}\hat{\lambda}(X_{H_{s,t}}(u) dt - du)$, 
but we can show the opposite inequality by exactly same arguments as in the proof of Lemma \ref{lem:convexity}.
\end{proof}

\subsection{Truncated Floer homology of Liouville domains}
In section 2.2.1, we define truncated Floer homology of admissible Hamiltonians, and introduce the \textit{monotonicity homomorphism}.
In section 2.2.2, we define truncated Floer homology of Liouville domains, by taking a direct limit with respect to the monotonicity homomorphism.

\subsubsection{Truncated Floer homology of admissible Hamiltonians}
Let $(X,\lambda)$ be a Liouville domain, 
$H \in \mca{H}_{\ad}(X,\lambda)$, $\alpha$ be a homotopy class of free loops on $X$, 
and $I \subset \R$ be a nonempty interval. 
We develop our arguments under the following assumption:
\begin{quote}
We are given $\Z$-valued Conley-Zehnder index $\ind_{\CZ}: \mca{P}(H) \to \Z$. 
\end{quote}
\begin{rem}
In general (when $c_1(TX) \ne 0$), it is impossible to assign $\Z$-valued Conley-Zehnder index to periodic orbits of Hamiltonian vector fields on $(X,\lambda)$.
Even when $c_1(TX)=0$, there is no canonical way to assign $\Z$-valued Conley-Zehnder index to noncontractible orbits. 
However, when $(X,\lambda)=(DT^*M,\lambda_M)$ where $M$ is an oriented Riemannian manifold, there exists a canonical way 
to assign $\Z$-valued Conley-Zehnder index to periodic orbits of Hamiltonian vector fields on $(X,\lambda)$
(see Lemma 1.2, Lemma 1.3 in \cite{AS1}). 
In the proof of our main result Theorem \ref{thm:main}, 
we only deal with the case $(X,\lambda)=(DT^*M,\lambda_M)$, where $M$ is a closed 
oriented Riemannian manifold.  
Hence, in this paper, it is enough to develop our arguments under the above assumption. 
\end{rem}

For each integer $k$,
$\CF^{I,\alpha}_k(H)$ denotes a free $\Z_2$ module generated over
\[
\{ x \in \mca{P}(H) \mid \ind_{\CZ}(x)=k, \mca{A}_H(x) \in I, [x]=\alpha \},
\]
where $\mca{A}_H(x)$ is defined by 
\[
\mca{A}_H(x):= \int_{S^1} x^*\hat{\lambda} - H_t\bigl(x(t)\bigr)\, dt. 
\]
$\CF^\alpha_k(H)$ abbreviates $\CF^{\R,\alpha}_k(H)$.

\begin{rem}
Throughout this paper, we work on $\Z_2$ - coefficient homology.
\end{rem}

Let $J=(J_t)_{t \in S^1}$ be a family of elements in $\mca{J}(\hat{X},\hat{\lambda})$, which is of contact type at $\infty$. 
For $x,y \in \mca{P}(H)$, define $\hat{\mca{M}}(x,y)$ as follows
($u(s)$ donotes $S^1 \to \hat{X}; t \mapsto u(s,t)$):
\begin{align*}
\hat{\mca{M}}(x,y):=\bigl\{u:\R \times S^1 \to \hat{X} &\mid \partial_s u - J_t\bigl(\partial_t u - X_{H_t}(u)\bigr)=0, \\
&\qquad \lim_{s \to -\infty} u(s) =x, \, \lim_{s \to \infty} u(s)=y \bigr\}.
\end{align*}
Notice that one can define a natural $\R$ action on $\hat{\mca{M}}(x,y)$ by shifting trajectories in the $s$-variable. 
$\mca{M}(x,y)$ denotes its quotient $\hat{\mca{M}}(x,y)/\R$.
For generic $J=(J_t)_{t \in S^1}$, the following holds: 
$\mca{M}(x,y)$ is a smooth manifold of dimension $\ind_\CZ(x) - \ind_{\CZ}(y)-1$.
When $\dim \mca{M}(x,y)=0$, $\mca{M}(x,y)$ is compact (hence is a finite point set).
Moreover, 
\[
\partial_{H,J}: \CF^\alpha_*(H) \to \CF^\alpha_{*-1}(H); \qquad [x] \mapsto \sum_y \sharp \mca{M}(x,y) \cdot [y]
\]
satisfies $\partial_{H,J}^2=0$. 
These claims are proved by the usual transversality and glueing arguments, combined with a $C^0$-estimate for solutions of the Floer equation
(Lemma \ref{lem:convexity}). 
The homology group of $(\CF^\alpha_*(H), \partial_{H,J})$ does not depend on $J$, and is denoted as $\HF^\alpha_*(H)$. 
It is called \textit{Floer homology} of $H$. 

It is easy to check that for any $x,y \in \mca{P}(H)$ and $u \in \hat{\mca{M}}(x,y)$, there holds 
\[
\mca{A}_H(x) - \mca{A}_H(y) = \int_{\R \times S^1} |\partial_s u(s,t)|_{J_t}^2 \, dsdt \ge 0.
\]
In particular, $\mca{M}(x,y) \ne \emptyset \implies \mca{A}_H(x) \ge \mca{A}_H(y)$. 
Hence for any nonempty interval $I \subset \R$, $\bigl(\CF_*^{I,\alpha}(H), \partial_{H,J})$ is a chain complex. 
Its homology group $\HF^{I,\alpha}_*(H)$ is called \textit{truncated Floer homology} of $H$. 
We introduce some abbreviations:
\[
\HF^{<a,\alpha}_*(H):= \HF^{(-\infty,a),\alpha}_*(H), \quad
\HF^I_*(H):= \bigoplus_\alpha \HF^{I,\alpha}_*(H), \quad
\HF_*(H):=\HF^\R_*(H).
\]

We define $I_+, I_- \subset \R$ as
$I_+:=(-\infty,\inf I\,]\cup I$, 
$I_-:=I_+ \setminus I$.
The following statements are immediate from the definition: 
\begin{itemize}
\item 
For any nonempty intervals $I, I' \subset \R$ such that $I_\pm \subset I'_\pm$, 
there exists a natural homomorphism $\Phi^{II'}_H:\HF^{I,\alpha}_*(H) \to \HF^{I',\alpha}_*(H)$.
\item
For any $-\infty \le a<b<c \le \infty$, there holds the following long exact sequence:
\[
\xymatrix{
\cdots\ar[r]&\HF_*^{[a,b),\alpha}(H)\ar[r]&\HF_*^{[a,c),\alpha}(H)\ar[r]&\HF_*^{[b,c),\alpha}(H)\ar[r]&\HF_{*-1}^{[a,b),\alpha}(H)\ar[r]&\cdots \\
}.
\]
\end{itemize}

Next we introduce the \textit{monotonicity homomorphism}. 

\begin{prop}\label{prop:monotonicity}
Let $H, H' \in \mca{H}_{\ad}(X,\lambda)$ and assume that $a_H \le a_{H'}$.
Notice that
\[
\Delta:=\int_{S^1} \max (H_t - H'_t) \, dt < \infty.
\]
Let $I, I' \subset \R$ be nonempty intervals which satisfy
$I_\pm + \Delta \subset I'_\pm$, and $\alpha$ be a homotopy class of free loops on $X$. 
Then, there exists a natural homomorphism 
$\Phi^{II'}_{HH'}:\HF_*^{I,\alpha}(H) \to \HF_*^{I',\alpha}(H')$.
Moreover, there holds the following properties:
\begin{itemize}
\item When $H=H'$, $\Phi^{II'}_{HH}$ coincides with $\Phi^{II'}_H$.
\item Suppose that $H, H', H'' \in \mca{H}_{\ad}(X)$ satisfy $a_H \le a_{H'} \le a_{H''}$, and 
let $I, I', I'' \subset \R$ be nonempty intervals. 
Then, the following diagram commutes if the homomorphisms $\Phi^{II'}_{HH'}$, $\Phi^{II''}_{HH''}$, $\Phi^{I'I''}_{H'H''}$ are defined. 
\[
\xymatrix{
&\HF_*^{I',\alpha}(H')\ar[rd]^{\Phi^{I'I''}_{H'H''}}& \\
\HF_*^{I,\alpha}(H)\ar[ru]^{\Phi^{II'}_{HH'}}\ar[rr]_{\Phi^{II''}_{HH''}}&&\HF_*^{I'',\alpha}(H'')
}.
\]
\end{itemize}
\end{prop}
\begin{proof} 
The proof is almost same as the case of closed aspherical symplectic manifolds (see \cite{Schwarz} pp.431).
The only difference is that we need a $C^0$-estimate for Floer trajectories, and it follows from Lemma \ref{lem:convexity}.
\end{proof}

The homomorphism $\Phi^{II'}_{HH'}$ defined in Proposition \ref{prop:monotonicity} is called the \textit{monotonicity homomorphism}.
The following corollary is immediate from Proposition \ref{prop:monotonicity}.

\begin{cor}\label{cor:monotonicity}
Let $(X,\lambda)$ be a Liouville domain, $\alpha$ be a homotopy class of free loops on $X$, 
and $H, H' \in \mca{H}_\ad(X,\lambda)$. 
\begin{enumerate}
\item If $a_H=a_{H'}$, there exists a natural isomorphism $\Phi_{HH'}: \HF^\alpha_*(H) \to \HF^\alpha_*(H')$.
\item If $H_t \le H'_t$ for any $t \in S^1$, there exists a natural homomorphism 
$\HF^{I,\alpha}_*(H) \to \HF^{I,\alpha}_*(H')$ for any nonempty interval $I$.
\end{enumerate}
\end{cor}

\subsubsection{Truncated Floer homology of Liouville domains}

Let $(X,\lambda)$ be a Liouville domain, 
$\alpha$ be a homotopy class of free loops on $X$, and 
$I \subset \R$ be a nonempty interval. 
Setting $\mca{H}^\neg_\ad(X,\lambda):=\{ H \in \mca{H}_\ad(X,\lambda) \mid  H|_{S^1 \times X} <0 \}$, 
we define 
\[
\HF^{I,\alpha}_*(X,\lambda):= \varinjlim_{H \in \mca{H}^\neg_\ad(X,\lambda)} \HF^{I,\alpha}_*(H),
\]
where the right hand side is a direct limit with respect to monotonicity homomorphisms in Corollary \ref{cor:monotonicity} (2).
If two nontrivial intervals $I, I'$ satisfy $I_\pm \subset I'_\pm$, 
there exists a natural homomorphism $\HF^{I,\alpha}_*(X,\lambda) \to \HF^{I',\alpha}_*(X,\lambda)$.
We prove the following useful lemma. 

\begin{lem}\label{lem:truncated}
For any $H \in \mca{H}_{\ad}(X,\lambda)$, there exists a natural isomorphism 
$\Psi_H: \HF^\alpha_*(H) \to \HF^{<a_H,\alpha}_*(X,\lambda)$. 
Moreover, if $H_-, H_+ \in \mca{H}_{\ad}(X,\lambda)$ satisfy $a_{H_-} \le a_{H_+}$, 
the following diagram commutes:
\[
\xymatrix{
\HF^\alpha_*(H_-) \ar[r]^-{\Psi_{H_-}} \ar[d]_{\Phi_{H_-H_+}} & \HF_*^{<a_{H_-},\alpha}(X,\lambda) \ar[d] \\
\HF^\alpha_*(H_+) \ar[r]_-{\Psi_{H_+}} & \HF_*^{<a_{H_+},\alpha}(X,\lambda)
}.
\]
\end{lem}
\begin{proof}
First we construct $\Psi_H: \HF^\alpha_*(H) \to \HF^{<a_H,\alpha}_*(X,\lambda)$. 
It is not hard to check that the following natural homomorphisms are all isomorphic:
\begin{align*}
\HF^\alpha_*(H) &\to \varinjlim_{\substack{G \in \mca{H}_\ad \\ a_G \le a_H}} \HF^\alpha_*(G), \\
\varinjlim_{\substack{G \in \mca{H}^\neg_\ad \\ a_G \le a_H}} \HF^\alpha_*(G) &\to 
\varinjlim_{\substack{G \in \mca{H}_\ad \\ a_G \le a_H}} \HF^\alpha_*(G), \\
\varinjlim_{\substack{G \in \mca{H}^\neg_\ad \\ a_G \le a_H}} \HF_*^{<a_H,\alpha}(G) &\to
\varinjlim_{\substack{G \in \mca{H}^\neg_\ad \\ a_G \le a_H}} \HF_*^\alpha (G), \\
\varinjlim_{\substack{G \in \mca{H}^\neg_\ad \\ a_G \le a_H}} \HF_*^{<a_H,\alpha}(G) &\to 
\varinjlim_{G \in \mca{H}^\neg_\ad} \HF_*^{<a_H,\alpha}(G)=\HF_*^{<a_H,\alpha}(X,\lambda). \\
\end{align*}
By composing the above isomorphisms and their inverses, we get an isomophsim
$\Psi_H: \HF^\alpha_*(H) \to \HF^{<a_H,\alpha}_*(X,\lambda)$. This proves the first assertion.
The second assertion follows from the above construction.
\end{proof}

It is a standard fact that for any $\delta \in \bigl(0, \min \Spec(X,\lambda)\bigr)$, 
there exists a natural isomorphism 
$H^{n-*}(X) \cong \HF^{<\delta}_*(X,\lambda)$. 
(see \cite{Viterbo}, Proposition 1.4). 
Then, for any $0<a \le \infty$, one can define a natural homomorphism 
\[
\iota_a: H^{n-*}(X) \cong \HF_*^{<\delta}(X,\lambda) \to \HF_*^{<a}(X,\lambda) 
\]
by taking sufficiently small $\delta>0$.
Using $\iota_a$, 
we define an important homology class $F_a \in \HF_n^{<a}(X,\lambda)$ by 
\[
F_a:= \iota_a(1). 
\]

For any $H \in \mca{H}_{\ad}(X,\lambda)$, we define 
$F_H \in \HF_n(H)$ by 
$F_H:= \Psi_H^{-1}(F_{a_H})$. 
The second assertion in Lemma \ref{lem:truncated} shows that 
for any $H_-, H_+ \in \mca{H}_{\ad}(X,\lambda)$ with $a_{H_-} \le a_{H_+}$, 
there holds $\Phi_{H_-H_+}(F_{H_-})=F_{H_+}$.

\subsection{Product structure}

First we define the \textit{pair-of-pants} Riemannian surface $\Pi$.
The following definition is taken from \cite{AS2}, pp.1602--1603.
In the disjoint union $\R \times [-1,0] \sqcup \R \times [0,1]$, we consider the identifications 
\begin{align*}
(s,-1) \sim (s,0^{-}), \quad  (s,0^+) \sim (s,1)   \qquad &(s \le 0) , \\
(s,0^-) \sim (s,0^+), \quad (s,-1) \sim (s,1)  \qquad &(s \ge 0), 
\end{align*}
and define $\Pi$ to be the quotient. 
We define the standard complex structure at every point on $\Pi$ other than $P:=(0,0) \sim (0,-1) \sim (0,1)$.
On a neighborhood of $P$, we define a complex structure by the following holomorphic coordinate:
\[
(\star):\qquad \bigl\{ \zeta \in \C \bigm{|}  |\zeta| < 1/\sqrt{2} \bigr\} \to \Pi; \quad
\zeta=\sigma+\tau i \mapsto 
\begin{cases}
(\sigma^2-\tau^2,2\sigma\tau) &( \sigma \ge 0), \\
(\sigma^2-\tau^2,2\sigma\tau+1) &( \sigma \le 0, \tau \ge 0), \\
(\sigma^2-\tau^2,2\sigma\tau-1) &( \sigma \le 0, \tau \le 0).
\end{cases}
\]
$j_{\Pi}$ denotes the complex structure on $\Pi$.
We need the following convexity result:

\begin{lem}\label{lem:convexity2}
Let $(H_{s,t})_{(s,t) \in \Pi}$ be a family of Hamiltonians on $\hat{X}$, 
and 
$(J_{s,t})_{(s,t) \in \Pi}$ be a family of elements in $\mca{J}(\hat{X},\hat{\lambda})$.
Suppose that there exists $r_0 >0$, such that the following holds:
\begin{itemize}
\item There exist $a,b \in C^\infty(\R)$ such that 
$H_{s,t}(z,r)=a(s)r+b(s)$ on $\partial X \times [r_0,\infty)$, and $a'(s) \ge 0$ for any $s \in \R$.
\item For any $(s,t) \in \Pi$, $J_{s,t}$ is of contact type on $\partial X \times [r_0, \infty)$.
\end{itemize}
If $u: \Pi \to \hat{X}$ satisfies the Floer equation 
\begin{align*}
\partial_s u- J_{s,t}(\partial_t u- X_{H_{s,t}}(u)) =0 \qquad &(\text{at $(s,t) \ne P$}) \\
J_P \circ du \circ j_{\Pi} - du=0                            \qquad &(\text{at $P$})
\end{align*}
and $u^{-1}(\partial X \times (r_0,\infty))$ is bounded, then $u(\Pi) \subset X(r_0)$. 
\end{lem}
\begin{proof}
If $u(\Pi)$ is not contained in $X(r_0)$, then there exists $r_1>r_0$ such that 
$u(\Pi)$ is not contained in $X(r_1)$, $u$ is transversal to $\partial X \times \{r_1\}$ and 
$u(P) \notin \partial X \times \{r_1\}$. 
Then, $D:=u^{-1}\bigl(\partial X \times [r_1,\infty)\bigr)$ is a compact surface with boundary, and $P \notin \partial D$. 

The rest part of the proof is almost same as that of Lemma \ref{lem:convexity}, replacing $D$ with $D \setminus \{P\}$. 
The only delicate point is that we have to check
\[
\int_{D \setminus \{P\}} \partial_s\bigl(\hat{\lambda}(X_{H_{s,t}}(u)) \bigr)\, ds dt - u^*(d\hat{\lambda})
= \int_{\partial D} \hat{\lambda}(X_{H_{s,t}}(u)\, dt - du) 
\]
holds true, where we cannot apply Stokes's theorem (when $P \in \interior D$, $D \setminus \{P\}$ is noncompact).
It is enough to consider the case $P \in \interior D$. 
Take a complex chart $(\star)$ near $P$, and set 
$D_\varepsilon:=\{\zeta \in \C \mid 0 \le |\zeta| \le \varepsilon \}$.
Then, the above identity is proved as 
\begin{align*}
&\int_{D \setminus \{P\}} \partial_s\bigl(\hat{\lambda}(X_{H_{s,t}}(u)) \bigr)\, ds dt - u^*(d\hat{\lambda})
=\lim_{\varepsilon \to 0} \int_{D \setminus D_\varepsilon} \partial_s\bigl(\hat{\lambda}(X_{H_{s,t}}(u)) \bigr)\, ds dt - u^*(d\hat{\lambda}) \\
&\qquad =\lim_{\varepsilon \to 0} \int_{\partial D_\varepsilon \cup \partial D} \hat{\lambda}(X_{H_{s,t}}(u)\,dt - du) 
=\int_{\partial D} \hat{\lambda}(X_{H_{s,t}}(u)\, dt- du).
\end{align*}
\end{proof}

Let $H, K \in \mca{H}_\ad(X,\lambda)$. Suppose that the following holds:
\begin{enumerate}
\item[(P0):] $\partial_t^rH|_{t=0} = \partial_t^rK|_{t=0}$ for any integer $r \ge 0$. In particular, $a_H=a_K$.
\end{enumerate}
We define $H*K \in C^\infty(S^1 \times \hat{X})$ by 
\[
(H * K)_t:= \begin{cases} 2H_{2t} &( 0 \le t \le 1/2) \\ 2K_{2t-1} &(1/2 \le t \le 1) \end{cases}.
\]
Suppose also that the following holds:
\begin{enumerate}
\item[(P1):] $H*K \in \mca{H}_\ad(X)$, 
\item[(P2):] For any $x \in \mca{P}(H)$ and $y \in \mca{P}(K)$, $x(0) \ne y(0)$. 
\end{enumerate}

Let $(J_t)_{-1 \le t \le 1}$ be a family of elements in $\mca{J}(\hat{X},\hat{\lambda})$, which is of contact type at $\infty$ and 
$\partial_t^r J_t|_{t=-1} = \partial_t^r J_t|_{t=0} = \partial_t^r J_t|_{t=1}$ for any integer $r \ge 0$. 
For any $x \in \mca{P}(H)$, $y \in \mca{P}(K)$ and $z \in \mca{P}(H*K)$, 
let $\mca{M}(x,y:z)$ denote the set of $u: \Pi \to \hat{X}$ which satisfies 
\begin{align*}
\partial_s u- J_t(\partial_t u- X_{(H*K)_{(t+1)/2}}(u)) =0 \qquad &(\text{at $(s,t) \ne P$}) \\
J_P \circ du \circ j - du=0                            \qquad &(\text{at $P$})
\end{align*}
with boundary conditions 
\begin{align*}
&\lim_{s \to -\infty} u(s,t) = x(t) \quad (0 \le t \le 1), \\
&\lim_{s \to -\infty} u(s,t) = y(t+1) \quad (-1 \le t \le 0), \\
&\lim_{s \to \infty} u(s,t) = z\bigl((t+1)/2\bigr) \quad( -1 \le t \le 1).
\end{align*}
For generic $(J_t)_{-1 \le t \le 1}$, following holds:
$\mca{M}(x,y:z)$ is a smooth manifold with dimension $\ind_{\CZ}(x)+\ind_{\CZ}(y) - \ind_{\CZ}(z)-n$. 
When $\dim \mca{M}(x,y:z)=0$, $\mca{M}(x,y:z)$ is compact (hence is a finite point set). Moreover, 
\[
\CF_i(H) \otimes \CF_j(K) \to \CF_{i+j-n}(H*K): \quad 
[x] \otimes [y] \mapsto \sum_z \sharp \mca{M}(x,y:z) \cdot  [z]
\]
is a chain map. These claims are proved by the usual transversality and glueing arguments, combined with a $C^0$-estimate
for Floer trajectories, which follows from Lemma \ref{lem:convexity2}.
Hence we can define the \textit{pair-of-pants product} on Floer homology of Hamiltonians:
\[
\HF_i(H) \otimes \HF_j(K) \to \HF_{i+j-n}(H*K); \quad \alpha \otimes \beta \mapsto \alpha*\beta.
\]

Simple computations show that for any $x \in \mca{P}(H)$, $y \in \mca{P}(K)$, $z \in \mca{P}(H*K)$ and 
$u \in \mca{M}(x,y:z)$, 
\[
\mca{A}_H(x) + \mca{A}_K(y) - \mca{A}_{H*K}(z)=\int_{\Pi \setminus \{P\}} |\partial_s u|_{J_t}^2 \, dsdt \ge 0.
\]
In particular, $\mca{M}(x,y:z) \ne \emptyset \implies \mca{A}_H(x) + \mca{A}_K(y) \ge \mca{A}_{H*K}(z)$.
Hence for any $-\infty \le a,b \le \infty$, one can define the pair-of-pants product on truncated Floer homology of Hamiltonians: 
\[
\HF^{<a}_i(H) \otimes \HF^{<b}_j(K) \to \HF^{<a+b}_{i+j-n}(H*K).
\]
By using Lemma \ref{lem:convexity2}, it is easy to show that it commutes with monotonicity homomorphisms: 

\begin{lem}\label{lem:product}
Suppose that $H, K, \bar{H}, \bar{K} \in \mca{H}_\ad(X,\lambda)$ satisfy the following: 
\begin{enumerate}
\item $H$ and $K$ satisfy (P0), (P1), (P2).
\item $\bar{H}$ and $\bar{K}$ satisfy (P0), (P1), (P2).
\item $H_t \le \bar{H}_t$, $K_t \le \bar{K}_t$ for any $t \in S^1$. 
\end{enumerate}
Then, the following diagram commutes for any $-\infty \le a,b \le \infty$: 
\[
\xymatrix{
\HF^{<a}_i(H) \otimes \HF^{<b}_j(K) \ar[r]\ar[d]_{\Phi_{H\bar{H}} \otimes \Phi_{K\bar{K}}} & \HF^{<a+b}_{i+j-n}(H*K)\ar[d]^{\Phi_{H*K,\bar{H}*\bar{K}}} \\
\HF^{<a}_i(\bar{H}) \otimes \HF^{<b}_j(\bar{K}) \ar[r] & \HF^{<a+b}_{i+j-n}(\bar{H}*\bar{K})
}.
\]
\end{lem}

By Lemma \ref{lem:product}, one can define the 
pair-of-pants product on truncated Floer homology of Liouville domains: 
$\HF^{<a}_i(X,\lambda) \otimes \HF^{<b}_j(X,\lambda) \to \HF^{<a+b}_{i+j-n}(X,\lambda)$.

Moreover, the isomorphism in Lemma \ref{lem:truncated} commutes with products. 
More precisely, if $H, K \in \mca{H}_\ad(X,\lambda)$ satisfy (P0), (P1), (P2), 
the following diagram commutes ($a:=a_H=a_K$):
\[
\xymatrix{
\HF_i(H) \otimes \HF_j(K) \ar[r]\ar[d]_{\Psi_H \otimes \Psi_K} & \HF_{i+j-n} (H*K) \ar[d]^{\Psi_{H*K}} \\
\HF_i^{<a}(X,\lambda) \otimes \HF_j^{<a}(X,\lambda) \ar[r] & \HF_{i+j-n}^{<2a}(X,\lambda)
}.
\]

\section{Floer homology of cotangent bundles and loop product}
The computation of Floer homology of cotengent bundles has been studied by several authors 
 (see \cite{AS1}, \cite{AS2}, \cite{SW}, \cite{Viterbo}).
We mainly follow \cite{AS1}, \cite{AS2}. 
The aim of this section is to state Theorem \ref{thm:loop}, which relates Floer homology of cotangent bundles
 with singular homology of loop spaces, and 
the pair-of-pants product with the loop product. 

In section 3.1, we recall the definition of the loop product following \cite{AS2}.
In section 3.2, we state Theorem \ref{thm:loop}. 
Although Theorem \ref{thm:loop} is essentially established in \cite{AS1} and \cite{AS2}, 
it requires some technical arguments to deduce it from results in those papers in a rigorous manner. 
Since these arguments are rather technical, they are postponed until section 6.

\subsection{Loop product}
In this subsection, we recall the definition of the loop product, which was discovered in \cite{CS}.
The following exposition is taken from section 1.2 in \cite{AS2}
(although authors work on $C^0(S^1,M)$ rather than $W^{1,2}(S^1,M)$ there). 

First we recall the definition of the \textit{Umkehr map}.
Let $X$ be a Hilbert manifold, $Y$ be its closed submanifold with codimension $n$, and $i:Y \to X$ be the inclusion map. 
Let $NY$ denote the normal bundle of $Y$. 
The tublar neighborhood theorem (see \cite{Lang}, IV, sections 5-6)
claims that there exists a unique (up to isotopy) open embedding 
$u: NY \to X$ such that $u(y,0)= i(y)$ for any $y \in Y$. 
Setting $U:= u(NY)$, the \textit{Umkehr map} $i_!: H_*(X) \to H_{*-n}(Y)$ is defined as 
\[
\xymatrix{
H_*(X)\ar[r]&H_*(X, X \setminus Y) \ar[r]^{\cong}&H_*(U,U \setminus Y) \ar[r]^{(u_*)^{-1}}&H_*(NY, NY \setminus Y) \ar[r]^-{\tau}&H_{*-n}(Y).
}
\]
The second arrow is the isomorphism given by excision, the last one is the Thom isomorphism associated to the vector bundle $NY \to Y$
(although it is not oriented, we can consider the Thom isomorphism since we are working on $\Z_2$-coefficient homology).
For later purposes, we state the following lemma which is immediate from the above definition:

\begin{lem}\label{lem:umkehr}
Let $X$ be a Hilbert manifold, $Y$ be its closed submanifold with codimension $n$, and $i:Y \to X$ be the inclusion map. 
Let $Z$ be a $k$-dimensional compact manifold, and $f: Z \to X$ be a smooth map which is transversal to $Y$. Then, there holds
the following identity in $H_{k-n}(Y)$ : 
$i_{!}\bigl(f_*[Z]\bigr) = f_*\bigl[f^{-1}(Y)\bigr] $.
\end{lem}

Then we define the loop product.
Let $M$ be a $n$-dimensional Riemannian manifold, and $a,b>0$. 
Let us consider the evaluation map $\ev \times \ev: \Lambda^{<a}_M \times \Lambda^{<b}_M \to M \times M; (\gamma,\gamma') \mapsto \bigl(\gamma(0), \gamma'(0)\bigr)$ and 
the diagonal $\Delta_M:=\{(x,x) \in M \times M \mid x \in M\} \subset M \times M$.
Then,
$\Theta^{a,b}_M:=(\ev \times \ev)^{-1}(\Delta_M)$
is a $n$-codimensional submanifold of $\Lambda^{<a}_M \times \Lambda^{<b}_M$. 
Let us denote the embedding map $\Theta^{a,b}_M \to \Lambda^{<a}_M \times \Lambda^{<b}_M$ by $e$. 
Moreover, $\Gamma:\Theta^{a,b}_M \to \Lambda^{<a+b}_M$ denotes the concatenation map, i.e.
\[
\Gamma(\gamma,\gamma')(t):= \begin{cases}
                            \gamma(2t) &( 0 \le t \le 1/2) \\
                            \gamma'(2t-1) &(1/2 \le t \le 1)
                            \end{cases}.
\]
Then, the loop product 
\[
\circ: H_i(\Lambda^{<a}_M) \otimes H_j(\Lambda^{<b}_M) \to H_{i+j-n}(\Lambda^{<a+b}_M)
\]
is defined as the composition of the following three homomorphisms:
\[
\xymatrix{
H_i(\Lambda^{<a}_M) \otimes H_j(\Lambda^{<b}_M) \ar[r]^{\times}&H_{i+j}(\Lambda^{<a}_M \times \Lambda^{<b}_M) \ar[r]^{e_!}&H_{i+j-n}(\Theta^{a,b}_M) \ar[r]^{\Gamma_*} &H_{i+j-n}(\Lambda^{<a+b}_M).
}
\]
The first arrow is the usual cross-product in singular homology.
The second one is the Umkehr map associated to the embedding $e: \Theta^{a,b}_M \to \Lambda^{<a}_M \times \Lambda^{<b}_M$, 
and the last one is induced by the concatenation map $\Gamma$. 

\subsection{Floer homology of cotangent bundles}

We state Theorem \ref{thm:loop}. 
As we have explained at the beginning of this section, its proof is postponed until section 6:

\begin{thm}\label{thm:loop}
For any closed oriented $n$-dimensional Riemannian manifold $M$, there holds the following statements:
\begin{enumerate}
\item For any homotopy class $\alpha$ of free loops on $M$ and 
$0<a \le \infty$, there exists a natural isomorphism $\HF_*^{<a,\alpha}(DT^*M) \cong H_*(\Lambda^{<a,\alpha}_M)$.
\item For any $0<a \le b \le \infty$, the following diagram commutes:
\[
\xymatrix{
\HF^{<a,\alpha}_*(DT^*M) \ar[r]^-{\cong}\ar[d] & H_*(\Lambda^{<a,\alpha}_M)\ar[d] \\
\HF^{<b,\alpha}_*(DT^*M) \ar[r]_-{\cong}& H_*(\Lambda^{<b,\alpha}_M)
}.
\]
The right arrow is induced by the inclusion $\Lambda^{<a,\alpha}_M \to \Lambda^{<b,\alpha}_M$.
\item The following diagram commutes:
\[
\xymatrix{
\HF_i(DT^*M) \otimes \HF_j(DT^*M)\ar[r]\ar[d]_{\cong}&\HF_{i+j-n}(DT^*M)\ar[d]^{\cong} \\
H_i(\Lambda_M) \otimes H_j(\Lambda_M) \ar[r]& H_{i+j-n}(\Lambda_M)
}.
\]
The top arrow is the pair-of-pants product, the bottom arrow is the loop product.
\end{enumerate}
\end{thm}

Here are two corollaries: 

\begin{cor}\label{cor:loop}
For any $0<a \le \infty$, the isomorphism $\HF^{<a}_*(DT^*M) \cong H_*(\Lambda^{<a}_M)$ maps $F_a \in \HF^{<a}_n(DT^*M)$ to 
$c_*[M] \in H_n(\Lambda^{<a}_M)$, where $c:M \to \Lambda^{<a}_M$ maps $p \in M$ to the constant loop at $p$.
\end{cor}
\begin{proof}
Theorem \ref{thm:loop} (2) implies that it is enough to prove the assertion for sufficiently small $a>0$.
We may assume that $M$ is connected. 
When $a>0$ is sufficiently small, both $\HF^{<a}_n(DT^*M)$ and $H_n(\Lambda^{<a}_M)$ are isomorphic to $\Z_2$, and they are generated by 
$F_a$ and $c_*[M]$, respectively. Hence the assertion is obvious in this case. 
\end{proof}

\begin{cor}\label{cor:loop2}
For any $0<a \le \infty$, $F_a \ne 0$. 
\end{cor}
\begin{proof}
$c_*: H_*(M) \to H_*(\Lambda^{<a}_M)$ is injective, since $\ev_*: H_*(\Lambda^{<a}_M) \to H_*(M)$ is a left inverse of $c_*$. 
Hence $c_*[M] \ne 0$. Therefore Corollary \ref{cor:loop} implies $F_a \ne 0$. 
\end{proof}

\section{Spectral invariants} 
The notion of \textit{spectral invariants} has been developed by several authors
(see \cite{FGS}, \cite{MS} section 12.4, and references therein). 
In this section, we define the spectral invariants for admissible Hamiltonians on Liouville domains.  
In section 4.1, we define the spectral invariants and summerize their basic properties. 
In section 4.2, we establish a criterion for Hamiltonians to have nonconstant periodic orbits (Proposition \ref{prop:criterion}), 
 which is used in the proof of our main result, Theorem \ref{thm:main}.

\subsection{Definition and basic properties}
Let $(X,\lambda)$ be a Liouville domain, 
and $H \in \mca{H}_\ad(X,\lambda)$.
For any $a \in \R$, there exists a long exact sequence 
\[
\xymatrix{
\cdots \HF_*^{<a}(H) \ar[r]^-{i^a_*} & \HF_*(H) \ar[r]^-{j^a_*} & \HF^{\ge a}_*(H) \ar[r]^-{\partial^a_*} &\HF^{<a}_{*-1}(H) \cdots.
}
\]
Recall that there exists a natural isomorphism 
$\Psi_H:\HF_*(H) \to \HF^{<a_H}_*(X,\lambda)$.
Then, for any $x \in \HF^{<a_H}_*(X,\lambda)$, we define the \textit{spectral invariant} $\rho(H:x)$ by 
\[
\rho(H:x):= \inf\{a \in \R  \mid \Psi_H^{-1}(x) \in \Im i^a_* \} = \inf \{a \in \R \mid j^a_*(\Psi_H^{-1}(x))=0 \}.
\]
Notice that $\rho(H:0)=-\infty$.
We abbreviate $\rho(H: F_{a_H})$ as $\rho(H)$.

In the next Lemma \ref{lem:spectral}, 
we summerize basic properties of the spectral invariants. 
First we introduce some notations:
\begin{itemize}
\item 
For $H \in C^\infty(S^1 \times \hat{X})$ and a homotopy class $\alpha$ of free loops on $X$, 
we define 
\begin{align*}
\mca{P}^\alpha(H)&:= \{ x \in \mca{P}(H) \mid [x] = \alpha \}, \qquad
\Spec^\alpha(H):= \{ \mca{A}_H(x) \mid x \in \mca{P}^\alpha(H) \}, \\
\Spec(H)&:= \{ \mca{A}_H(x) \mid x \in \mca{P}(H) \}. 
\end{align*}
It is easy to see that for any $H \in C^\infty(S^1 \times \hat{X})$ which is linear at $\infty$ and $a_H \notin \Spec(X,\lambda)$, 
$\Spec^\alpha(H), \Spec(H) \subset \R$ are closed, nowhere dence sets.
\item For $H \in C_0^\infty(S^1 \times \hat{X})$, 
its \textit{Hofer norm} is defined as
\[ 
\|H\|:= \int_{S^1} \max H_t - \min H_t \, dt.
\]
\end{itemize}
\begin{lem}\label{lem:spectral}
\begin{enumerate}
\item For any $H \in \mca{H}_{\ad}(X,\lambda)$ and $x \in \HF^{<a_H,\alpha}_*(X,\lambda) \setminus \{0\}$,  $\rho(H:x) \in \Spec^\alpha(H)$.
\item Suppose $H, K \in \mca{H}_{\ad}(X,\lambda)$ satisfy that $\supp (H-K)$ is compact. 
Then, for any $x \in \HF_*^{<a_H}(X,\lambda) \setminus \{0\}$, $|\rho(H:x) - \rho(K:x)| \le \| H- K\|$.
\item Suppose $H, K \in \mca{H}_{\ad}(X,\lambda)$ satisfy (P0), (P1), (P2) in section 2.3.
Then, for any $x,y \in \HF_*^{<a_H}(X,\lambda)$, 
$\rho(H*K: x*y) \le \rho(H:x) + \rho(K:y)$.
\end{enumerate}
\end{lem}
\begin{proof}
(1): Suppose the contrary: i.e. $\rho(H:x) \notin \Spec^\alpha(H)$. 
We abbreviate $\rho(H:x)$ as $\rho$. 
Since $x \ne 0$, $\rho \ne -\infty$. 
Since $\Spec^\alpha(H)$ is closed, there exists $\varepsilon>0$ such that 
$[\rho-\varepsilon, \rho+\varepsilon] \cap \Spec^\alpha(H)=\emptyset$.
Hence $\HF_*^{[\rho-\varepsilon,\rho+\varepsilon),\alpha}(H)=0$, therefore 
$\HF_*^{<\rho-\varepsilon,\alpha}(H) \to \HF_*^{<\rho+\varepsilon,\alpha}(H)$ is isomorphic. 
Hence we get 
\[
\Im \bigl( \HF_*^{<\rho-\varepsilon,\alpha}(H) \to \HF_*^\alpha(H) \bigr) 
=\Im \bigl( \HF_*^{<\rho+\varepsilon,\alpha}(H) \to \HF_*^\alpha(H) \bigr). 
\]
However, the definition of the spectral invariant implies that 
$\Psi_H^{-1}(x) \notin \Im \bigl( \HF_*^{<\rho-\varepsilon,\alpha}(H) \to \HF_*^\alpha(H) \bigr) $ and 
$\Psi_H^{-1}(x) \in \Im \bigl( \HF_*^{<\rho+\varepsilon,\alpha}(H) \to \HF_*^\alpha(H) \bigr)$, hence a contradiction.

(2): By Proposition \ref{prop:monotonicity}, for any $a \in \R$ there exists a monotonicity homomorphism 
$\HF_*^{<a}(H) \to \HF_*^{<a+\|H-K\|}(K)$.
The commutativity of the diagram (horizontal arrows are monotonicity homomorphisms)
\[
\xymatrix{
\HF_*^{<a}(H)\ar[r]\ar[d]_{i^a_*}&\HF_*^{<a+\|H-K\|}(K) \ar[d]^{i^{a+\|H-K\|}_*}\\
\HF_*(H)\ar[r]&\HF_*(K)
}
\]
shows that $\rho(K:x) \le \rho(H:x)+\|H-K\|$. A similar argument shows that $\rho(H:x) \le \rho(K:x)+\|H-K\|$, hence (2) is proved. 

(3) follows from the fact that the isomorphism in Lemma \ref{lem:truncated} commutes with products (see the last paragraph of section 2.3). 
\end{proof}

Using Lemma \ref{lem:spectral} (2), we can define the spectral invariants for larger class of Hamiltonians. 
Let $H \in C^\infty(S^1 \times \hat{X})$ be linear at $\infty$ and $a_H \notin \Spec(X,\lambda)$, 
but $\mca{P}(H)$ may contain degenerate orbits. 
Take a sequence $(H_j)_{j=1,2,\ldots}$ of admissible Hamiltonians so that
$\supp (H_j - H)$ is compact for any $j$ and  
$\lim_{j \to \infty} \|H_j - H\| =0$. Define $\rho(H:x)$ by 
\[
\rho(H:x):= \lim_{j \to \infty} \rho(H_j:x).
\]
By Lemma \ref{lem:spectral} (2), the right hand side converges and does not depend on choices of $(H_j)_j$. 
The following lemma is immediate from Lemma \ref{lem:spectral} and the above definition.

\begin{lem}\label{lem:spectral2}
Let $H \in C^\infty(S^1 \times \hat{X})$ be linear at $\infty$ and $a_H \notin \Spec(X,\lambda)$.
\begin{enumerate}
\item For any $x \in \HF_*^{<a_H,\alpha}(X,\lambda) \setminus \{0\}$, $\rho(H:x) \in \Spec^\alpha(H)$.
\item For any $K \in C^\infty(S^1 \times \hat{X})$ such that $\supp (H-K)$ is compact and for any $x \in \HF_*^{<a_H}(X,\lambda) \setminus \{0\}$, 
$|\rho(H:x) - \rho(K:x)| \le \| H- K\|$.
\item Suppose that $2a_H \notin \Spec(X,\lambda)$. If $K \in C^\infty(S^1 \times \hat{X})$ satisfies 
$\partial_t^rH|_{t=0}=\partial_t^rK|_{t=0}$ for any integer $r \ge 0$, 
$\rho(H*K:x*y) \le \rho(H:x)+\rho(K:y)$ for any $x,y \in \HF^{<a_H}_*(X,\lambda)$. 
\end{enumerate}
\end{lem}

\subsection{Criterion for Hamiltonians to have nonconstant periodic orbits}

In this subsection, we prove the following criterion for Hamiltonians to have nonconstant periodic orbits. 
Recall that for any $H \in \mca{H}_\ad(X,\lambda)$,  $\rho(H:F_{a_H})$ is abbreviated as $\rho(H)$. 

\begin{prop}\label{prop:criterion}
Let $(X,\lambda)$ be a Liouville domain, 
$H \in C_0^\infty(\interior X)$, 
$\chi \in C^\infty(S^1)$, and 
$\nu \in C^\infty\bigl([1,\infty)\bigr)$.
Define $H_{\chi,\nu} \in C^\infty(S^1 \times \hat{X})$ by 
\[
H_{\chi,\nu}(t,x) := \begin{cases}
                   \chi(t) H(x) &( x \in X) \\
                   \nu(r) &\bigl(x=(z,r) \in \partial X \times [1,\infty) \bigr)
                   \end{cases}.
\]
Suppose that the following holds:
\begin{enumerate}
\item $\nu \in C^\infty\bigl([1,\infty)\bigr)$ satisfies the following properties:
\begin{enumerate}
\item There exists $1<r_0$ such that $\nu(r) \equiv 0$ on $[1,r_0]$.
\item There exist $1<r_1$ and $a_\nu \in (0,-\min H) \setminus \Spec(X,\lambda)$ 
such that $\nu'(r) \equiv a_\nu$ on $[r_1,\infty)$.
\item $S(\nu):=\sup_{r \ge 1}\,r\nu'(r)-\nu(r)< - \min H$.
\end{enumerate}
\item $\int_{S^1} \chi(t) dt=1$.
\item $F_{a_\nu} \ne 0$. 
\end{enumerate}
Under these assumptions, if  $\rho(H_{\chi,\nu}) \ne - \min H$, then there exists a 
nonconstant periodic orbit $\gamma$ of $X_H$ with $\Per(\gamma) \le 1$. 
\end{prop}

First we prove the following lemma. 
For any $K \in C^\infty(\hat{X})$, we define $\bar{K} \in C^\infty(S^1 \times \hat{X})$ by 
$\bar{K}(t,x):=K(x)$. 

\begin{lem}\label{lem:Morse}
Fix $r>1$. 
Let $K$ be a Morse function on $\hat{X}$, such that 
$\bar{K} \in \mca{H}_\ad(X,\lambda)$ and 
$K$ is linear on $\partial X \times [r,\infty)$. 
If $C^2$-norm of $K|_{X(r)}$ is sufficiently small, 
$\rho(\bar{K})=-\min K$. 
\end{lem}
\begin{proof}
If $C^2$-norm of $K|_{X(r)}$ is sufficiently small,
the Floer complex of $\bar{K}$ is identified with 
the Morse complex of $K$, and it induces an isomorphism 
$\HF_*(\bar{K}) \cong H^{n-*}(X)$. 
Since $F_{\bar{K}} \in \HF_n(\bar{K})$ corresponds to $1 \in H^0(X)$ in this isomorphism, 
$\sum_{q \in \CrP_0(K) }a_q [q]$ represents $F_{\bar{K}}$ if and only if 
$a_q=1$ for any $q \in \CrP_0(K):=\{ \text{critical points of $K$ with Morse index $0$}\}$. 
Hence $\rho(\bar{K}) = \max_{q \in \CrP_0(K)} -K(q) = -\min K$. 
\end{proof}

Now we return to the proof of Proposition \ref{prop:criterion}.
It is enough to show the following claim: 
\begin{claim}
Suppose that assumptions (1), (2), (3) in Proposition \ref{prop:criterion} hold. 
If every nonconstant periodic orbit of $X_H$ has period strictly larger than $1$, 
then $\rho(H_{\chi,\nu})=-\min H$.
\end{claim}
The proof consists of $4$ steps. 
In the course of the proof, we use the following notation: 
for any $a \in \R$, we also denote the constant function $S^1 \to \R; t \mapsto a$ by $a$. 
For instance, $H_{a,\nu}$ denotes the function on $S^1 \times \hat{X}$ defined as
\[
H_{a,\nu}(t,x):=\begin{cases} aH(x) &(x \in X) \\ \nu(r) &\bigl(x=(z,r) \in \partial X \times [1,\infty)\bigr) \end{cases}.
\]

\textbf{Step 1:}
There exist $\varepsilon_0, \delta_0 \in (0,1)$ such that: for any 
$\varepsilon \in (0,\varepsilon_0]$ and $\delta \in (0,\delta_0]$ such that $\delta a_\nu \notin \Spec(X,\lambda)$, 
$\rho(H_{\varepsilon,\delta\nu})= - \varepsilon \min H$. 

For any $\varepsilon,\delta >0$, there exists a sequence of Morse functions $(H_j)_{j=1,2,\ldots}$ on $\hat{X}$ such that 
$\supp (H_{\varepsilon,\delta\nu}-\bar{H_j})$ is contained in $X(r_1)$, 
and $\lim_{j \to \infty} \bar{H_j}= H_{\varepsilon,\delta\nu}$ in $C^2$-norm.
When $\varepsilon$ and $\delta$ are sufficiently small,
$C^2$-norm of $H_{\varepsilon,\delta\nu}|_{X(r_1)}$ is sufficiently small. 
Hence
\[
\rho(H_{\varepsilon,\delta\nu}) = \lim_{j \to \infty} \rho(\bar{H_j}) = \lim_{j \to \infty} - \min H_j = - \min H_{\varepsilon,\delta\nu}
=-\varepsilon \min H,
\]
where the second equality follows from Lemma \ref{lem:Morse}.

\textbf{Step 2:} For any $0<\delta<\min \biggl\{ \delta_0, \frac{-\varepsilon_0 \min H}{S(\nu)} \biggr\}$ such that $\delta a_\nu \notin \Spec(X,\lambda)$, 
 $\rho(H_{1,\delta\nu})=-\min H$.

For any $\varepsilon \in (0,1]$, $\mca{P}(\varepsilon H)$ consists of only constant loops at critical points of $H$, 
since every nonconstant periodic orbit of $X_H$ has period $>1$. 
On the otherhand, for any $x \in \mca{P}(H_{\varepsilon,\delta\nu})$ which is not contained in $X$, 
$\mca{A}_{H_{\varepsilon,\delta\nu}}(x) \le \delta S(\nu)$.
Hence $\Spec (H_{\varepsilon,\delta\nu}) \subset \bigl(-\infty,\delta S(\nu)\bigr] \cup -\varepsilon \CrV(H)$, where 
$\CrV(H)$ denotes the set of critical values of $H$. 
Since $\delta a_\nu \le a_\nu$ and $F_{a_\nu} \ne 0$, $F_{\delta a_\nu} \ne 0$. 
Hence Lemma \ref{lem:spectral2} (1) shows that 
$\rho(H_{\varepsilon,\delta\nu}) \in \bigl(-\infty,\delta S(\nu)\bigr] \cup -\varepsilon \CrV(H)$.

Let $I:=\{\varepsilon \in [\varepsilon_0,1] \mid \rho(H_{\varepsilon,\delta\nu})=-\varepsilon \min H\}$. 
Step 1 shows $\varepsilon_0 \in I$. 
Lemma \ref{lem:spectral2} (2) shows that $\rho(H_{\varepsilon,\delta\nu})$ depends continuously on $\varepsilon$, hence $I$ is closed. 
Moreover, since $\delta S(\nu) < -\varepsilon_0 \min (H)$ and $\CrV(H)$ is nowhere dence, $I$ is open.
Hence $I=[\varepsilon_0,1]$. In particular, $\rho(H_{1,\delta\nu})=-\min H$.

\textbf{Step 3:} $\rho(H_{1,\nu})=-\min H$.

Since $F_{a_\nu} \ne 0$, 
\[
\rho(H_{1,\nu}) \in \Spec (H_{1,\nu}) \subset  \bigl(-\infty,S(\nu)\bigr] \cup -\CrV(H) \subset (-\infty, -\min H].
\]
Hence $\rho(H_{1,\nu}) \le -\min H$.
We prove the opposite inequality.
Take $\delta>0$ so that $\delta a_\nu \notin \Spec(X,\lambda)$ and $0<\delta <\min \biggl\{ \delta_0, \frac{-\varepsilon_0 \min H}{S(\nu)} \biggr\}$.
For sufficiently small $c>0$, there exist $G^-, G^+ \in \mca{H}_\ad(X,\lambda)$ with the following properties:
\begin{enumerate}
\item[(a):] $G^+ - H_{1,\nu}$ and $G^- - H_{1,\delta\nu}$ are compactly supported. 
\item[(b):] $\| G^+ - H_{1,\nu}\|, \| G^- - H_{1,\delta\nu}\| <c$.
\item[(c):] $G^+ = G^-$ on $S^1 \times X(r_0)$.
\item[(d):] $G^+(t,x) \ge G^-(t,x)$ for any $t \in S^1$, $x \in \hat{X}$.
\item[(e):] There exist $a,b \in \R$ such that $G^\pm(t,z,r)=ar+b$ for any $(t,z,r) \in S^1 \times \partial X \times [1,r_0]$.
\item[(f):] Any $x \in \mca{P}(G^\pm)$ satisfying $\mca{A}_{G^\pm}(x) \ge -\min H-c$ is contained in $X$. 
\end{enumerate}
Consider the following commutative diagram:
\[
\xymatrix{
\HF_*(G^-) \ar[r]\ar[d] &\HF_*^{\ge -\min H - c}(G^-) \ar[d] \\
\HF_*(G^+) \ar[r] & \HF_*^{\ge -\min H-c}(G^+) 
},
\]
where vertical arrows are monotonicity homomorphisms. 
(d), (e), (f) and Lemma \ref{lem:convexity3} imply that all Floer trajectories involved in the 
right arrow are contained in $X$.
Hence (c) shows that the right arrow is an isomorphism. 

By Step 2, $\rho(H_{1,\delta\nu})=-\min H$. 
Hence (a), (b) and Lemma \ref{lem:spectral2} (2) show $\rho(G^-) > -\min H - c$. 
It implies that $F_{G^-}$ does not vanish by the top arrow. 
Since the right arrow is isomorphic, 
$F_{G^+}$ does not vanish by the bottom arrow, i.e. $\rho(G^+) \ge -\min H-c$.
Finally, (a), (b) and Lemma \ref{lem:spectral2} (2) show $\rho(H_{1,\nu}) > -\min H - 2c$.
Since $c>0$ is arbitrarily small, 
$\rho(H_{1,\nu}) \ge -\min H$.

\textbf{Step 4:} $\rho(H_{\chi,\nu})=-\min H$.

For any $s \in [0,1]$, define $\chi_s$ by 
$\chi_s:= (1-s) + s\chi$.
Then, it is easy to check that $\Spec(H_{\chi_s,\nu})$ does not depend on $s$. 
Moreover, Lemma \ref{lem:spectral2} (2) shows that $\rho(H_{\chi_s,\nu})$ depends continuously on $s$.
Since $\Spec(H_{\chi_s,\nu})$ is nowhere dence, $\rho(H_{\chi_s,\nu})$ does not depend on $s$. 
Hence $\rho(H_{\chi,\nu}) = \rho(H_{1,\nu}) = -\min H$.
\qed

\section{Proof of Theorem \ref{thm:main}}
\subsection{Key lemma}
First we introduce a few notations. 
For any loop $\gamma: S^1 \to M$, we define 
$\bar{\gamma}:S^1 \to M$ by 
$\bar{\gamma}(t):=\gamma(1-t)$. 
Since $\gamma \sim \gamma' \implies \bar{\gamma} \sim \bar{\gamma'}$, 
one can define $\bar{\alpha}$ for any homotopy class $\alpha$ of free loops. 

\begin{lem}\label{lem:key}
Let $M$ be a closed Riemannian manifold.
Suppose that $\ev: \Lambda_M \to M; \gamma \mapsto \gamma(0)$ has a smooth section $s$. 
Then, setting $\bar{s}: M \to \Lambda_M$ and $c: M \to \Lambda_M$ by 
\[
\bar{s}(p):= \overline{s(p)}, \qquad 
c(p):= \text{constant loop at $p$} \qquad( \forall p \in M),
\]
there holds $s_*[M] \circ \bar{s}_*[M] = c_*[M]$, where $\circ$ denotes the loop product on $H_*(\Lambda_M)$.
\end{lem}
\begin{proof}
Recall that the loop product is the compsition of the following three homomorphisms
(we use same notations as in section 3.1):
\[
\xymatrix{
H_i(\Lambda_M) \otimes H_j(\Lambda_M) \ar[r]^{\times} & H_{i+j}(\Lambda_M \times \Lambda_M) \ar[r]^{e_!} & H_{i+j-n}(\Theta_M) \ar[r]^{\Gamma_*}&H_{i+j-n}(\Lambda_M).
}
\]
It is clear that $s_*[M] \times \bar{s}_*[M] = (s \times \bar{s})_*[M \times M]$.
Moreover, since $(\ev \times \ev) \circ (s \times \bar{s}) = \id_{M \times M}$ and 
$\Theta_M=(\ev \times \ev)^{-1}(\Delta_M)$, 
$s \times \bar{s}: M \times M \to \Lambda_M \times \Lambda_M$ is transversal to $\Theta_M$.
Then Lemma \ref{lem:umkehr} shows that 
$s_*[M] \circ \bar{s}_* [M] = \Gamma(s,\bar{s})_*[M]$.

Hence it is enough to show that two continuous maps
$\Gamma(s,\bar{s}), c: M \to \Lambda_M$ are homotopic.
Define 
$K: M \times [0,1] \to \Lambda_M$ by 
\[
K(p,t)(\tau):= \begin{cases}
                     s(p)(2t\tau) &( 0 \le \tau \le 1/2) \\
                     s(p)(2t-2t\tau) &(1/2 \le \tau \le 1)
                     \end{cases}.
\]
Then $K$ is a homotopy between $\Gamma(s,\bar{s})$ and $c$. 
\end{proof}

Lemma \ref{lem:key}, Theorem \ref{thm:loop} and Corollary \ref{cor:loop} imply the following:

\begin{cor}\label{cor:key}
Let $M$ be a closed, oriented Riemannian manifold, 
$\alpha$ be a homotopy class of free loops on $M$. 
Suppose that the evaluation map $\Lambda^{\alpha}_M \to M; \gamma \mapsto \gamma(0)$ has a smooth section. 
Then, there exist $x \in \HF_n^{\alpha}(DT^*M)$, $y \in \HF_n^{\bar{\alpha}}(DT^*M)$ such that 
$x*y = F_\infty \in \HF_n(DT^*M)$.
\end{cor}

\subsection{Proof of Theorem \ref{thm:main}}
Finally we prove our main result, Theorem \ref{thm:main}.
Since $\HF_*(DT^*M)= \varinjlim_{a \to \infty} \HF_*^{<a}(DT^*M)$, 
Corollary \ref{cor:key} shows that, for sufficiently large $0<a<\infty$, there exist
$x \in \HF_n^{<a,\alpha}(DT^*M)$, 
$y \in \HF_n^{<a,\bar{\alpha}}(DT^*M)$ such that 
$x*y = F_{2a} \in \HF_n^{<2a}(DT^*M)$. 
Corollary \ref{cor:loop2} shows that $F_{2a} \ne 0$, hence $x,y \ne 0$. 

We show the following claim: 
\begin{claim}\label{claim:main}
For any $H \in C_0^\infty(\interior DT^*M)$ such that $\min H <-2a$, 
there exists a nonconstant periodic orbit $\gamma$ of $X_H$ such that $\Per(\gamma) \le 1$.
\end{claim}
Claim \ref{claim:main} immediately implies that $c_{\HZ}(DT^*M) \le 2a$. In particular, $c_{\HZ}(DT^*M)$ is finite. 
Suppose that Claim \ref{claim:main} does not hold: 
i.e. there exists $H \in C_0^\infty(\interior\,DT^*M)$ such that 
$\min H < -2a$ and any nonconstant periodic orbit $\gamma$ of $X_H$ satisfies $\Per(\gamma)>1$. 

Since $\Spec (DT^*M, \lambda_M)$ is nowhere dence, we may assume that $a, 2a \notin \Spec(DT^*M, \lambda_M)$. 
Take $\nu \in C^\infty\bigl([1,\infty)\bigr)$ such that:
\begin{enumerate}
\item[(a)] There exists $1<r_0$ such that $\nu(r) \equiv 0$ on $[1,r_0]$.
\item[(b)] There exists $1<r_1$ such that
such that $\nu'(r) \equiv a$ on $[r_1,\infty)$.
\item[(c)] $S(\nu):=\sup_{r \ge 1}\,r\nu'(r)-\nu(r)< - \min H/2$.
\end{enumerate}
Moreover, take $\chi \in C^\infty(S^1)$ such that $\int_{S^1} \chi(t) dt=1$ and $\chi \equiv 0$ near $0=1$.

Consider $H_{\chi,\nu}, H_{0,\nu} \in C^\infty(S^1 \times T^*M)$. Obviously, 
$H_{\chi,\nu} * H_{0,\nu}= H_{\eta,2\nu}$, where $\eta \in C^\infty(S^1)$ is defined as 
\[
\eta(t) = \begin{cases}
          2\chi(2t) &( 0 \le t \le 1/2) \\
           0       &( 1/2 \le t \le 1)
          \end{cases}.
\]
By the same arguments as in Step 4 in the proof of Proposition \ref{prop:criterion}, 
$\rho(H_{\chi,\nu}:x) = \rho(H_{1,\nu}:x)$. 
If there exists $\gamma \in \mca{P}(H_{1,\nu})$ with $[\gamma]=\alpha$ and $\gamma(S^1) \subset DT^*M$,  
$\gamma$ is a nonconstant (since $\alpha$ is nontrivial) periodic orbit of $X_H$, and obviously $\Per(\gamma)=1$. 
It contradicts our assumption. 
Hence any $\gamma \in \mca{P}(H_{1,\nu})$ with $[\gamma]=\alpha$ is not contained in $DT^*M$, and 
(c) shows that 
$\Spec^\alpha(H) \subset (-\infty,-\min H/2)$.
Hence $\rho(H_{\chi,\nu}:x)= \rho(H_{1,\nu}:x) < - \min H/2$.

On the otherhand, (c) shows that $\Spec (H_{0,\nu}) \subset (-\infty, -\min H/2)$.
Hence $\rho(H_{0,\nu}:y)<- \min H/2$.
Therefore we conclude
\[
\rho(H_{\eta,2\nu})= \rho(H_{\eta,2\nu}:F_{2a}) \le \rho(H_{\chi,\nu}:x) + \rho(H_{0,\nu}:y) < -\min H.
\]
Since $2a, S(2\nu) < -\min H$, 
$\int_{S^1} \eta(t) \,dt=1$ and $F_{2a} \ne 0$, 
Proposition \ref{prop:criterion} shows that there exists a nonconstant periodic orbit of $X_H$ with 
period $\le 1$. This contradicts our assumption. 
\qed

\section{Proof of Theorem \ref{thm:loop}}

The aim of this section is to prove Theorem \ref{thm:loop}.
Although it is essentially established in \cite{AS1}, \cite{AS2}, 
we have to overcome the following technical matters to deduce it from results in \cite{AS1}, \cite{AS2}:
\begin{itemize}
\item In section 2, we define truncated Floer homology of Liouville domains, by using Hamiltonians which grow linearly at ends. 
On the otherhand, in \cite{AS1} and \cite{AS2}, the authors study Floer homology of Hamiltonians on cotangent bundles which grow quadratically at ends. 
Hence we have to understand how to define truncated Floer homology of unit disk cotangent bundles,
 by using Hamiltonians on cotangent bundles which grow quadratically at ends. 
\item In \cite{AS1}, the main result is stated as 
\[
\HF^{<a}_*(H) \cong H_*\bigl(\{\mca{A}_L<a\}\bigr),
\]
where $H$ is a time-dependent Hamiltonian on a cotangent bundle $T^*M$, 
$L$ is its Fenchel dual, and 
$\mca{A}_L$ is a functional on $\Lambda_M$ which is defined as 
\[
\mca{A}_L(\gamma):=\int_{S^1} L(t,\gamma(t),\gamma'(t))\, dt.
\]
On the otherhand, Theorem \ref{thm:loop} deals with truncated Floer homology of unit disk cotangent bundles,
which is defined by taking a limit of Hamiltonians. 
Hence to prove Theorem \ref{thm:loop}, we have to choose an appropriate sequence of Hamiltonians, and take a limit of the above isomorphism.
\end{itemize}
Section 6.1 concerns the first matter, and  
the goal of this subsection is to define truncated Floer homology of unit disk cotangent bundles 
by using Hamiltonians on cotangent bundles  which grow quadratically at ends (Proposition \ref{prop:HFinAS4}).
In section 6.2, we state main results in \cite{AS1}, \cite{AS2} in a rigorous manner (Theorem \ref{thm:AS}, Theorem \ref{thm:AS3}).
In section 6.3, we prove Theorem \ref{thm:loop}, 
by choosing an appropriate sequence of Hamiltonians (Lemma \ref{lem:sequence}) and 
taking a direct limit.

\subsection{Floer homology of cotangent bundles}
Let $M$ be a compact Riemannian manifold.
For each $(q,p) \in T^*M$, 
$T_{(q,p)}^{\hor}(T^*M)$ denotes the horizontal subspace of $T_{(q,p)}(T^*M)$ with respect to the Levi-Civita connection, and 
$T_{(q,p)}^{\ver}(T^*M):=T_{(q,p)}(T_q^*M)$. 
$T(T^*M)$ naturally splits as $T^\hor(T^*M) \oplus T^\ver(T^*M)$.
Then, natural isomorphisms of vector bundles 
\[
\xymatrix{
T^\ver(T^*M)\ar[r]^{\cong}&\pi_M^*(T^*M)\ar[r]^{\cong}&\pi_M^*(TM)\ar[r]^{\cong}&T^\hor(T^*M)
}
\]
define an almost complex structure on $T(T^*M)$ (the second arrow is defined by the Riemannian metric on $M$).
We denote it as $J_M$. 
It is easy to check that $J_M$ is compatible with $\omega_M$. 
Hence it defines a Riemannian metric $g_{J_M}$ on $T^*M$. 
Simple computations show the following lemma: 

\begin{lem}
Identifying $(T^*M,\lambda_M)$ with the completion of $(DT^*M,\lambda_M)$ (see example \ref{ex:TM}), 
$J_M$ is of contact type on $T^*M \setminus M = ST^*M \times (0,\infty)$.
\end{lem}
%
%

Let $H \in C^\infty(S^1 \times T^*M)$.
For any $t \in S^1$, $H_t \in C^\infty(T^*M)$ is defined as $H_t(q,p):=H(t,q,p)$.
In \cite{AS1}, the authors introduce the following conditions on $H$:
\begin{enumerate}
\item [(H0):] All orbits in $\mca{P}(H)$ are nondegenerate.
\item [(H1):] $dH_t(q,p)(Z_M) - H_t(q,p) \ge h_0|p|^2-h_1$ for some constants $h_0>0$ and $h_1 \ge 0$.
\item [(H2):] $|\nabla_q H_t(q,p)| \le h_2(1+|p|^2)$, $|\nabla_pH_t(q,p)| \le h_2(1+|p|)$ for some constant $h_2 \ge 0$.
\end{enumerate}

In (H1), $Z_M$ denotes the Liouville vector field of $(T^*M,\lambda_M)$, as we have introduced at the beginning of this paper. 
In (H2), $\nabla_qH_t$, $\nabla_pH_t$ denote horizontal and vertical components of the gradient of $H_t$ with respect to $g_{J_M}$. 

For any $-\infty \le a \le \infty$ and $H \in C^\infty(S^1 \times T^*M)$ satisfying (H0), (H1), (H2), 
one can define the Floer chain complex $(\CF^{<a}_*(H), \partial_{J_M,H})$ in the usual manner, 
and its homology group is denoted as $\HF^{<a}_*(H)$ (see section 3.1 in \cite{AS2}).
The crusial step is to prove a $C^0$-estimate for Floer trajectories, and it is carried out in section 1.5 in \cite{AS1}
(see also section 6.1 in \cite{AS2}).
\begin{rem}
In \cite{AS2}, the authors fix the almost complex structure $J_M$, and achieve transversalities by perturbing Hamiltonians 
(see Remark 3.3 in \cite{AS2}).
\end{rem}

We check the existence of the monotonicity homomorphism in this case: 

\begin{prop}\label{prop:HFinAS3}
Suppose that $H, K \in C^\infty(S^1 \times T^*M)$ satisfy (H0), (H1), (H2), and 
$H_t(q,p) \le  K_t(q,p)$ for any $t \in S^1$ and $(q,p) \in T^*M$. 
Then, there exists a natural homomorphism $\HF^{<a}_*(H) \to \HF^{<a}_*(K)$ for any $-\infty \le a \le \infty$. 
\end{prop}
\begin{proof}
We take $\chi \in C^\infty(\R)$ so that $\chi'(s) \ge 0$ for any $s \in \R$, 
$\chi(s)=0$ for $s \le -1$, $\chi(s)=1$ for $s \ge 1$. 
Set $H_{s,t}:=\chi(s)K_t + (1-\chi(s))H_t$.
Then we define a chain map 
$\varphi:\CF^{<a}_*(H) \to \CF^{<a}_*(K)$ by 
\[
\varphi [x]:=\sum_y \varphi_{x,y}  \cdot [y] 
\]
where $\varphi_{x,y}$ is the (mod 2) number of 
$u: \R \times S^1 \to T^*M$ which satisfies the Floer equation 
\[
\partial_s u - J_M(\partial_t u - X_{H_{s,t}}(u))=0, \qquad
\lim_{s \to -\infty} u(s)=x, \qquad
\lim_{s \to \infty} u(s)=y.
\]
The only delicate point is a $C^0$ estimate for solutions of the above Floer equation, where we cannot apply 
Lemma \ref{lem:convexity}, since $H_s$ are not linear at $\infty$. 
The key fact is that any solution $u:\R \times S^1 \to T^*M$ of the above Floer equation satisfies 
\[
\partial_s\bigl(\mca{A}_{H_s}\bigl(u(s)\bigr)\bigr) = -\int_{S^1} \big\lvert \partial_su(s,t) \big\rvert_{J_M}^2 + \partial_sH_{s,t}\bigl(u(s,t)\bigr)\, dt.
\]
Since $\partial_s H_{s,t}(q,p) \ge 0$ for any $(s,t) \in \R \times S^1$ and $(q,p) \in T^*M$, we get 
\begin{align*}
\int_{\R \times S^1} |\partial_s u(s,t)|_{J_M}^2 \, dsdt &\le \mca{A}_H(x) - \mca{A}_K(y), \\
\mca{A}_{H_s}\bigl(u(s)\bigr) &\in \bigl[ \mca{A}_K(y), \mca{A}_H(x) \bigr] \quad (\forall s \in \R).
\end{align*}
Then, Theorem 1.14 (i) in \cite{AS1} shows that 
$u(\R \times S^1)$ is contained in some compact set, which depends on $x$ and $y$
(in that theorem $s$-independent Hamiltonians are considered, however its proof makes use of only the above inequalities).
\end{proof}
We call the homomorphism in Proposition \ref{prop:HFinAS3} the \text{monotonicity homomorphism}.

One can define $\varinjlim_{H|_{S^1 \times DT^*M}<0} \HF^{<a}_*(H)$ for any $-\infty \le a \le \infty$, by taking a direct limit 
with respect to the monotonicity homomorphism in Proposition \ref{prop:HFinAS3}.
We have to check that it coincides with trucated Floer homology defined in section 2.2.2:

\begin{prop}\label{prop:HFinAS4}
For any $-\infty \le a \le \infty$, there exists a natural isomorphism 
\[
\HF^{<a}_*(DT^*M) \cong \varinjlim_{H|_{S^1 \times DT^*M}<0} \HF^{<a}_*(H),
\]
where the left hand side is truncated Floer homology defined in section 2.2.2, and 
the right hand side is a direct limit with respect to the monotonicity homomorphism.
\end{prop}
\begin{proof}
In this proof, we denote $(DT^*M,\lambda_M)$ as $(X,\lambda)$. 
Let $\mca{H}$ denote the set of $H \in C^\infty(S^1 \times \hat{X})$, which satisfies 
(H0), (H1), (H2) and $H|_{S^1 \times X}<0$.
Then the right hand side in the statement is written as $\varinjlim_{H \in \mca{H}} \HF^{<a}_*(H)$.
Notice that it is enough to consider the case $a<\infty$. 
Let $\mca{H}_a$ be the set of $H \in \mca{H}$, which satisfies the following properties:
\begin{itemize}
\item There exist $a_H \in (0,\infty) \setminus \Spec(X,\lambda)$, $b_H \in \R$ such that 
$H_t(z,r) = a_Hr+b_H$ for any $(z,r) \in \partial X \times [2,3]$.
\item Any $x \in \mca{P}(H)$ satisfying $\mca{A}_H(x)<a$ is contained in $X(2)$.
\end{itemize}
It is easy to see that $\mca{H}_a$ is cofinal in $\mca{H}$. 
Hence the natural homomorphism $\varinjlim_{H \in \mca{H}_a} \HF^{<a}_*(H) \to \varinjlim_{H \in \mca{H}} \HF^{<a}_*(H)$ is isomorphic.

For each $H \in \mca{H}_a$, we define $H^{\lin} \in \mca{H}_\ad(X,\lambda)$ by
\[
H^\lin_t(x):= \begin{cases}
               H_t(x) &\bigl( x \in X(2) \bigr), \\
               a_Hr+b_H   &\bigl( x=(z,r) \in \partial X \times [2,\infty) \bigr).
              \end{cases}
\]
Since any $x \in \mca{P}(H)$ with $\mca{A}_H(x)<a$ is contained in $X(2)$, 
there exists a natural isomorphism of modules $\CF^{<a}_*(H) \cong \CF^{<a}_*(H^\lin)$.
Since $J_M$ is of contact type on $\partial X \times (0,\infty)$ (hence on $\partial X \times [2,3]$), 
Lemma \ref{lem:convexity3} shows that the above isomorphism of modules is an isomorphism of complexes between
$(\CF^{<a}_*(H), \partial_{H,J_M})$ and $(\CF^{<a}_*(H^\lin), \partial_{H^\lin,J_M})$.
Hence we get an isomorphism $\HF^{<a}_*(H) \cong \HF^{<a}_*(H^\lin)$.

Suppose that $H, K \in \mca{H}_a$ satisfy $H_t(x) \le K_t(x)$ for any $t \in S^1$, $x \in \hat{X}$.
Then, Lemma \ref{lem:convexity3} shows that, if $b_H \ge  b_K$ 
\[
\xymatrix{
\HF^{<a}_*(H) \ar[r]^-{\cong}\ar[d] & \HF^{<a}_*(H^\lin) \ar[d] \\
\HF^{<a}_*(K) \ar[r]_-{\cong} & \HF^{<a}_*(K^\lin) 
}
\]
commutes, where vertical arrows are monotone homomorphisms.

It is easy to see that there exists a cofinal sequence $(H^i)_{i=1,2,\ldots}$ in $\mca{H}_a$, such that 
$H^1 \le H^2  \le \cdots$ and $b_{H^1} \ge b_{H^2} \ge \cdots$. 
Hence $\varinjlim_{H \in \mca{H}_a} \HF_*^{<a}(H) \cong \varinjlim_{H \in \mca{H}_a} \HF_*^{<a}(H^\lin)$.
On the otherhand, since $\{H^\lin\}_{H \in \mca{H}_a}$ is cofinal in $\{ H \in \mca{H}_{\ad}(X,\lambda) \mid H|_{S^1 \times X}<0\}$, 
the natural homomorphism $\varinjlim_{H \in \mca{H}_a} \HF_*^{<a}(H^\lin) \to \HF_*^{<a}(X)$ is isomorphic. 
Hence we end up with an isomorphism $\varinjlim_{H \in \mca{H}} \HF_*^{<a}(H) \cong \HF_*^{<a}(X)$.
\end{proof}

Suppose that $H,K \in C^\infty(S^1 \times T^*M)$ satisfy the following conditions:
\begin{itemize}
\item $\partial_t^rH|_{t=0}= \partial_t^rK|_{t=0}$ for any integer $r \ge 0$.
\item $H, K, H*K$ satisfy (H0), (H1) and (H2).
\end{itemize}
Under these assumptions, for any $-\infty \le a,b \le \infty$ 
one can define the pair-of-pants product
\[
\HF^{<a}_i(H) \otimes \HF^{<b}_j(K) \to \HF^{<a+b}_{i+j-n}(H*K)
\]
in the same manner as in section 2.3
(see section 3.3 in \cite{AS2}).
This product commutes with the monotonicity homomorphism defined in Proposition \ref{prop:HFinAS3}.
(The only delicate point is a $C^0$ estimate for solutions of the Floer equation, and it is proved by arguments 
similar to the proof of Proposition \ref{prop:HFinAS3}. See section 6.1 in \cite{AS2}.)
Hence one can define a product 
\[
\varinjlim_{H|_{S^1 \times DT^*M}<0} \HF^{<a}_i(H) \otimes
\varinjlim_{H|_{S^1 \times DT^*M}<0} \HF^{<b}_j(H) \to
\varinjlim_{H|_{S^1 \times DT^*M}<0} \HF^{<a+b}_{i+j-n}(H).
\]
The isomorphism in Proposition \ref{prop:HFinAS4} interwines this product with 
the pair-of-pants product on truncated Floer homology, which was defined in section 2.3.

\subsection{Main results in \cite{AS1}, \cite{AS2}}
In section 6.2.1, we construct a Morse complex of a Lagrangian action functional on a loop space, and see that 
its homology group is naturally isomorphic to singular homology of the loop space. 
In section 6.2.2,  we state the main results in \cite{AS1}, \cite{AS2} in a rigorous term. 

\subsubsection{Morse complexes on loop spaces}

Let $M$ be a Riemannian manifold, and $L \in C^\infty(S^1 \times TM)$.
For any $t \in S^1$, $L_t \in C^\infty(TM)$ is defined as $L_t(q,v):=L(t,q,v)$. 
In \cite{AS1}, the authors introduce the following conditions on $L$:
\begin{enumerate}
\item[(L1):] $\nabla_{vv}L_t(q,v) \ge l_1$ for some $l_1>0$. 
\item[(L2):] $|\nabla_{qq}L_t(q,v)| \le l_2(1+|v|^2)$, $|\nabla_{qv} L_t(q,v)| \le l_2(1+|v|)$, $|\nabla_{vv}L_t(q,v)| \le l_2$ for some $l_2 \ge 0$.
\end{enumerate}
The Lagrangian action functional $\mca{A}_L$ on $\Lambda_M$  is defined as 
\[
\mca{A}_L(\gamma):= \int_{S^1} L_t\bigl(\gamma(t), \gamma'(t) \bigr) \, dt, 
\]
and the corresponding Euler-Lagrange equation is 
\[
\frac{d}{dt} \partial_vL_t(\gamma(t), \gamma'(t))= \partial_qL_t(\gamma(t), \gamma'(t)).
\]
Let $\mca{P}(L)$ denote the set of solutions of the above Euler-Lagrange equation, and consider the following condition on $L$:
\begin{enumerate}
\item[(L0):] All orbits in $\mca{P}(L)$ are nondegenerate. 
\end{enumerate}

In \cite{AS3} (Theorem 4.1), it is proved that if $L \in C^\infty(S^1 \times TM)$ satisfies (L0), (L1) and (L2), then 
there exists a smooth vector field $X$ on $\Lambda_M$ with the following properties
(for terminologies, see \cite{AS3}):
\begin{enumerate}
\item[(X1):] $X$ is complete.
\item[(X2):] $\mca{A}_L$ is a Lyapnov function for $X$.
\item[(X3):] $X$ is Morse and all its singular points have finite Morse index.
\item[(X4):] The pair $(X, \mca{A}_L)$ satisfies the Palais-Smale condition. 
\item[(X5):] $X$ satisfies the Morse-Smale condition up to every order.
\end{enumerate}
Let $X$ be a smooth vector field on $\Lambda_M$ which satisfies (X1)-(X5). 
$\varphi_X: \Lambda_M \times \R \to \Lambda_M$ denotes the flow generated by $X$.
For any $x \in \mca{P}(L)$, $i(x:L)$ denotes the Morse index of $\mca{A}_L$ at $x$.
The \textit{stable} and \textit{unstable} mainifolds at $x \in \mca{P}(L)$ are defined as
\begin{align*}
W^s(x:X)&:=\{p \in \Lambda_M \mid \lim_{t \to \infty} \varphi_X(t,p)= x \}, \\
W^u(x:X)&:=\{p \in \Lambda_M \mid \lim_{t \to-\infty} \varphi_X(t,p)= x \}.
\end{align*}
They are submanifolds of $\Lambda_M$, and $\dim W^u(x:X)=i(x:L)$ (see Theorem 1.20 in \cite{AM}).
The following fact follows from results in \cite{AM} (see Corollary 4.1 in \cite{AS3}):
\begin{itemize}
\item For any $a \in \R$, let $\CM_*^{<a}(L)$ be a free $\Z_2$ module generated over $\bigl\{ x \in \mca{P}(L) \bigm{|} \mca{A}_L(x)<a, i(x:L)=* \bigr\}$.
Define $\partial_{L,X}: \CM_*^{<a}(L) \to \CM_{*-1}^{<a}(L)$ by 
\[
\partial_{L,X}[x]:= \sum_y \partial_{x,y} \cdot [y], 
\]
where $\partial_{x,y}$ is the (mod $2$) number of a compact $0$-dimensional manifold $W^u(x:X) \cap W^s(y:X)/\R$. 
Then, $(\CM_*^{<a}(L), \partial_{L,X})$ is a chain complex.
Its homology group $\HM_*^{<a}(L,X)$ is called \textit{Morse homology}.
\item $\HM_*^{<a}(L,X)$ is naturally isomorphic to singular homology $H_*\bigl(\{\mca{A}_L < a\}\bigr)$.
\end{itemize}

Suppose that $L^1, L^2 \in C^\infty(S^1 \times TM)$ 
satisfy (L0), (L1), (L2) and 
$L^1_t(q,v) \ge L^2_t(q,v)$ for any $t \in S^1$ and $(q,v) \in TM$. 
Then obviously $\{\mca{A}_{L^1}<a\} \subset \{\mca{A}_{L^2}<a\}$. We will describe a natural homomorphism 
$H_*\bigl(\{\mca{A}_{L^1}<a\}\bigr) \to H_*\bigl(\{\mca{A}_{L^2}<a\}\bigr)$ in terms of Morse homology, following 
Appendix A.2 in \cite{AS2}.

Take smooth vector fields $X^1, X^2$ on $\Lambda_M$ so that $X^i$ satisfies (X1)-(X5) with $L^i$, and assume that 
$\mca{P}(L^1) \cap \mca{P}(L^2) \cap \{\mca{A}_{L^1}<a\} = \emptyset$. 
Then, up to $C^\infty$-small perturbations of $X^1$ and $X^2$, we may assume that
\begin{quote}
For any $x \in \mca{P}(L^1) \cap \{\mca{A}_{L^1}<a\}$ and 
$y \in \mca{P}(L^2)$, 
$W^u(x:X^1)$ is transversal to $W^s(y,X^2)$.
\end{quote}
Then, one can define a chain map $\varphi: \CM_*^{<a}(L^1,X^1) \to \CM_*^{<a}(L^2,X^2)$ by 
\[
\varphi [x]: = \sum_y \varphi_{x,y} \cdot [y], 
\]
where $\varphi_{x,y}$ is the (mod $2$) number of $W^u(x:X^1) \cap W^s(y:X^2)$.
$\varphi$ induces a homomorphism on homology $\Phi: \HM_*^{<a}(L^1,X^1) \to \HM_*^{<a}(L^2,X^2)$, 
and $\Phi$ satisfies the following commutative diagram
(for proofs, see Appendix A.2 in \cite{AS2}): 
\[
\xymatrix{
\HM^{<a}_*(L^1,X^1) \ar[r]^-{\cong}\ar[d]_{\Phi} & H_*\bigl(\{\mca{A}_{L^1}<a\}\bigr) \ar[d] \\
\HM^{<a}_*(L^2,X^2) \ar[r]_-{\cong} & H_*\bigl(\{\mca{A}_{L^2}<a\}\bigr)
}.
\]

\subsubsection{Computations of Floer homology of cotangent bundles}

Suppose that $L \in C^\infty(S^1 \times TM)$ satisfies (L1), (L2).
(L1) implies that $L_t$ is strictly convex on each tangent fiber for any $t \in S^1$, 
hence one can define its Fenchel dual $H=(H_t)_{t \in S^1}$, namely 
\[
H_t(q,p):= \max_{v \in T_qM} p(v) - L_t(q,v) \qquad \bigl( (q,p) \in T^*M \bigr).
\]
There exists a $1:1$ correspondence between $\mca{P}(L)$ and $\mca{P}(H)$, and $L$ satisfies (L0) if and only if $H$ satisfies (H0).
The main result in \cite{AS1} is the following:
\begin{thm}\label{thm:AS}
Let $M$ be a compact oriented Riemannian manifold.
Suppose that $L \in C^\infty(S^1 \times TM)$ satisfies (L0), (L1), (L2), and
its Fenchel dual $H \in C^\infty(S^1 \times T^*M)$ satisfies (H0), (H1), (H2). 
\begin{enumerate}
\item For any $a>0$, there exists a natural isomorphism 
$H_*\bigl(\{\mca{A}_L<a\}\bigr) \cong \HF^{<a}_*(H)$. 
\item For any $a,b>0$ such that $a \le b$, the following diagram commutes:
\[
\xymatrix{
H_*\bigl(\{\mca{A}_L<a\}\bigr)\ar[r]^-{\cong}\ar[d] & \HF_*^{<a}(H) \ar[d] \\
H_*\bigl(\{\mca{A}_L<b\}\bigr)\ar[r]_-{\cong}& \HF_*^{<b}(H)
}.
\]
\end{enumerate}
\end{thm}
\begin{proof}
We only describe a construction of a chain level isomorphism to prove (1). 
Let $X$ be a smooth vector field on $\Lambda_M$ satisfying (X1)-(X5) with $L$. 
For $\gamma \in \mca{P}(L)$ and $x \in \mca{P}(H)$, consider the following moduli space:
\begin{align*}
\mca{M}(\gamma,x):=
\bigl\{ u: [0,\infty) \times S^1 \to T^*M &\bigm{|} \partial_s u- J_M(\partial_t u- X_H(u))=0, \\
& \qquad  \lim_{s \to \infty} u(s)=x, \,
\pi_M\bigl(u(0)\bigr) \in W^u(\gamma:X) \bigr\}.
\end{align*}
Following facts are proved in \cite{AS1}:
\begin{itemize}
\item Up to perturbations of $H$, $\mca{M}(\gamma,x)$ is a smooth manifold of dimension $i(\gamma:L)- \ind_{\CZ}(x)$. 
\item When $i(\gamma:L)-\ind_{\CZ}(x)=0$, $\mca{M}(\gamma,x)$ is compact (hence is a finite point set). We define $\psi:\CM_*(L,X) \to \CF_*(H,J_M)$
by
\[
\psi[\gamma]:= \sum_x \sharp \mca{M}(\gamma,x) \cdot [x].
\]
\item If $\mca{A}_H(x)>\mca{A}_L(\gamma)$ then $\mca{M}(\gamma,x)=\emptyset$.
Hence for any $a \in \R$, $\psi$ maps $\CM^{<a}_*(L:X)$ to $\CF^{<a}_*(H,J_M)$.
\item $\psi^{<a}: \CM^{<a}_*(L:X) \to \CF^{<a}_*(H,J_M)$ is isomorphic, and is a chain map. 
Hence it induces an isomorphism $\Psi^{<a}: \HM^{<a}_*(L:X) \to \HF_*^{<a}(H)$.
\end{itemize}
Combined with the natural isomorphism $\HM_*^{<a}(L:X) \cong H_*\bigl(\{\mca{A}_L<a\}\bigr)$, 
this proves the assertion (1). 
(2) is obvious from the construction. 
\end{proof}
\begin{rem}\label{rem:alpha}
It is clear from the above construction that for any homotopy class $\alpha$ of free loops on $M$, there exists a natural isomorphism 
\[
H_*\bigl(\{\gamma \in \Lambda_M \mid [\gamma]=\alpha, \mca{A}_L(\gamma)<a\}\bigr) \cong \HF_*^{<a,\alpha}(H).
\]
\end{rem}

It is easy to prove the following proposition from
the construction of the chain level isomorphsim in the proof of Theorem \ref{thm:AS} and 
results in section 6.2.1:

\begin{prop}\label{prop:AS2}
Let $M$ be a compact oriented Riemannian manifold.
Let $L^0, L^1 \in C^\infty(S^1 \times TM)$ satisfy (L0), (L1), (L2) and 
their Fenchel dual $H^0, H^1 \in C^\infty(S^1 \times T^*M)$ satisfy (H0), (H1), (H2). 
If $L^0_t(q,v) > L^1_t(q,v)$ for any $(t,q,v) \in S^1 \times TM$, then
$H^0_t(q,p) < H^1_t(q,p)$ for any $(t,q,p) \in S^1 \times T^*M$. 
Moreover, the following diagram commutes for any $a \in \R$: 
\[
\xymatrix{
H_*\big(\{\mca{A}_{L^0}<a\}\bigr)\ar[r]^-{\cong}\ar[d] & \HF_*^{<a}(H^0) \ar[d] \\
H_*\big(\{\mca{A}_{L^1}<a\}\bigr)\ar[r]_-{\cong} & \HF_*^{<a}(H^1)
}.
\]
The horizontal arrows are isomorphisms in Theorem \ref{thm:AS}, 
the left arrow is induced by the inclusion $\{\mca{A}_{L^0}<a\} \to \{\mca{A}_{L^1}<a\}$, and 
the right arrow is the monotonicity homomorphism defined in Proposition \ref{prop:HFinAS3}.
\end{prop}

Finally we state the main result in \cite{AS2}:

\begin{thm}\label{thm:AS3}
Let $M$ be a closed Riemannian manifold. 
Let $H, K \in C^\infty(S^1 \times T^*M)$ such that 
$\partial_t^rH|_{t=0}=\partial_t^rK|_{t=0}$ for any integer $r \ge 0$, and 
$H, K, H*K$ satisfy (H0), (H1), (H2). 
Then, the following diagram commutes: 
\[
\xymatrix{
H_i(\Lambda_M) \otimes H_j(\Lambda_M) \ar[r]\ar[d]_{\cong} & H_{i+j-n}(\Lambda_M) \ar[d]^{\cong}\\
\HF_i(H) \otimes \HF_j(K) \ar[r]& \HF_{i+j-n}(H*K).
}
\]
The vertical arrows are isomorphisms in Theorem \ref{thm:AS} (1) with $a=\infty$. 
The top arrow is the loop product, and the bottom arrow is the pair-of-pants product. 
\end{thm}
\subsection{Proof of theorem \ref{thm:loop}}
First we prove the following lemma:

\begin{lem}\label{lem:sequence}
Let $M$ be a closed Riemannian manifold.
Then, there exists a sequence of elements in $C^\infty(S^1 \times TM)$: 
$L^0>L^1> \cdots$, with the following properties:
\begin{enumerate}
\item For any integer $i$, $L^i$ satisfies (L0), (L1), (L2) and its Fenchel dual $H^i$ satisfies (H0), (H1), (H2).
\item For any $a>0$, $\bigcup_{i=1}^\infty \{\mca{A}_{L^i}<a\} = \Lambda^{<a}_M$.
\item For any $H \in C^\infty(S^1 \times T^*M)$ 
which satisfies (H1), (H2) and $H|_{S^1 \times DT^*M}<0$, 
$H<H^i$ for sufficiently large $i$. 
\end{enumerate}
\end{lem}
\begin{proof}
Since the nondegenerate conditions (L0), (H0) are achieved by $C^\infty$ small perturbations, 
it is enough to construct a sequence $(L^i)_i$ which satisfies (2), (3) and 
\begin{enumerate}
\item[(1)':] For any integer $i$, $L^i$ satisfies (L1), (L2) and its Fenchel dual $H^i$ satisfies (H1), (H2).
\end{enumerate}
Take a sequence $l^1>l^2> \cdots $ of $C^\infty([0,\infty))$ with the following properties:
\begin{itemize}
\item Each $l^i$ is an increasing and stricly convex function.
\item For each $l^i$, there exists $c>0$ such that $l^i(t)=l^i(0)+ ct^2$ for sufficiently small $t \ge 0$. 
\item For each $l^i$, there exists a quadratic function $q^i$ such that $\supp\,(l^i-q^i)$ is compact. 
\item For each $l^i$, $l^i(t) >t$ for any $t \in [0,\infty)$, 
\item If $l \in C^\infty([0,\infty))$ satisfies $l(t) > t$ for any $t \in [0,\infty)$ and 
there exists a quadratic function $q$ such that $\supp\,(l-q)$ is compact, 
$l>l_i$ for sufficently large $i$. 
\end{itemize}
Then, it is easy to check that a sequence $(L^i)_i$, defined by $L^i(t,q,v):= l^i(|v|)$ satisfies the condition.
\end{proof} 

Finally we prove Theorem \ref{thm:loop}.
Take a sequence $(L^i)_i$ as in Lemma \ref{lem:sequence} and let $(H^i)_i$ be its Fenchel dual. 
Then, Theorem \ref{thm:AS} and Proposition \ref{prop:AS2} imply that there exsits a natural isomorphism
\[
\varinjlim_{i \to \infty} H_*\bigl(\{\mca{A}_{L^i}<a\}\bigr) \cong \varinjlim_{ i \to \infty} \HF_*^{<a}(H^i).
\]
(2) in Lemma \ref{lem:sequence} shows that the left hand side is isomorphic to $H_*(\Lambda^{<a}_M)$. 
(3) in Lemma \ref{lem:sequence}  and Proposition \ref{prop:HFinAS4} show that the right hand side is isomorphic to $\HF_*^{<a}(DT^*M)$.
Hence we have obtained an isomorphism $H_*(\Lambda^{<a}_M) \cong \HF_*^{<a}(DT^*M)$.
Remark \ref{rem:alpha} shows that there exists an isomorphism 
$H_*(\Lambda^{<a,\alpha}_M) \cong \HF_*^{<a,\alpha}(DT^*M)$ for any homotopy class $\alpha$ of free loops on $M$. 
Theorem \ref{thm:loop} (2) follows from Theorem \ref{thm:AS} (2).
Theorem \ref{thm:loop} (3) follows from Theorem \ref{thm:AS3}. \qed

\section{Quantitative refinement of the main result}
It is likely that Theorem \ref{thm:loop} (3) can be refined in the following form: 
\begin{conj}
For any closed oriented Riemannian manifold $M$ and $0<a,b \le \infty$, the following diagram commutes:
\[
\xymatrix{
\HF^{<a}_i(DT^*M) \otimes \HF^{<b}_j(DT^*M)\ar[r]\ar[d]_{\cong}&\HF^{<a+b}_{i+j-n}(DT^*M)\ar[d]^{\cong} \\
H_i(\Lambda^{<a}_M) \otimes H_j(\Lambda^{<b}_M) \ar[r]& H_{i+j-n}(\Lambda^{<a+b}_M)
}.
\]
The vertical arrows are isomorphisms in Theorem \ref{thm:loop} (1), 
the top arrow is the pair-of-pants product, 
and the bottom arrow is the loop product. 
\end{conj}

Once this is established, same arguments as the proof of Theorem \ref{thm:main} imply the following quantitative refinement of Theorem \ref{thm:main}:

\begin{conj}
Let $M$ be a closed oriented Riemannian manifold, $\alpha$ be a nontrivial homotopy class of free loops on $M$, and $a>0$. 
If $\ev: \Lambda^{<a,\alpha} \to M$ has a smooth section, then $c_{\HZ}(DT^*M,\omega_M) \le 2a$. 
\end{conj}

\end{document}